\documentclass{article}

\usepackage[english]{babel}

\usepackage[letterpaper,top=2cm,bottom=2cm,left=3cm,right=3cm,marginparwidth=1.75cm]{geometry}

\usepackage{bm}
\usepackage{amsmath}
\usepackage{amssymb}
\usepackage{amsthm}
\usepackage{graphicx}
\usepackage{parallel}
\usepackage{ stmaryrd }
\usepackage{tikz}
\usepackage{multicol}
\usepackage{ stmaryrd }
\usepackage{ latexsym }
\usepackage[shortlabels]{enumitem}
\usepackage{hyperref}
\hypersetup{colorlinks,allcolors=blue,breaklinks=true}

\newcommand{\forkindep}[1][]{
  \mathrel{
    \mathop{
      \vcenter{
        \hbox{\oalign{\noalign{\kern-.3ex}\hfil$\vert$\hfil\cr
              \noalign{\kern-.7ex}
              $\smile$\cr\noalign{\kern-.3ex}}}
      }
    }\displaylimits_{#1}
  }
}

\newcommand{\mipo}{\operatorname{MiPo}}
\newcommand{\Char}{\operatorname{Char}}

\newcommand{\NN}{\mathbb{N}}

\newcommand{\QQ}{\mathbb{Q}}
\newcommand{\RR}{\mathbb{R}}
\newcommand{\MM}{\mathbb{M}}
\newcommand{\FF}{\mathbb{F}}

\newcommand{\VV}{\mathbb{V}}

\newcommand{\Kp}[1]{K[X]_{\operatorname{irr}}^{#1}}
\newcommand{\Cc}{{\mathcal{C}}}
\newcommand{\Ccalg}{{\mathcal{C}^{\operatorname{alg}}}}
\newcommand{\Cctrans}{{\mathcal{C}^{\operatorname{tr}}}}

\newcommand{\mm}{\mathcal{M}}

\newcommand{\ii}{\mathcal{I}}
\newcommand{\jj}{\mathcal{J}}
\newcommand{\kk}{\mathcal{K}}
\newcommand{\rr}{\mathcal{R}}
\newcommand{\dd}{\mathcal{D}}

\newcommand{\set}[1]{{\{#1\}}}
\newcommand{\Set}[1]{{\left\{#1\right\}}}
\newcommand{\spanA}[2]{{{\langle#1\rangle_{#2}}}}

\newcommand{\Fac}{{\operatorname{Fac}}}
\newcommand{\Ker}{{\operatorname{Ker}}}
\newcommand{\ACF}{{\operatorname{ACF}}}
\newcommand{\Image}{{\operatorname{Im}}}

\newcommand{\dcl}{{\operatorname{dcl}}}
\newcommand{\rk}{{\operatorname{rk}}}
\newcommand{\cl}{{\operatorname{cl}}}
\newcommand{\Diag}{{\operatorname{Diag}}}
\newcommand{\Id}{{\operatorname{Id}}}
\newcommand{\Th}{{\operatorname{Th}}}

\newcommand{\Lex}{{\operatorname{Lex}}}
\newcommand{\acl}{{\operatorname{acl}}}

\newcommand{\ua}{{\underline{a}}{}}
\newcommand{\ub}{{\underline{b}}{}}

\newcommand{\ud}{{\underline{d}}{}}
\newcommand{\ue}{{\underline{e}}{}}

\newcommand{\ut}{{\underline{t}}{}}
\newcommand{\uu}{{\underline{u}}{}}
\newcommand{\uv}{{\underline{v}}{}}
\newcommand{\uw}{{\underline{w}}{}}
\newcommand{\ux}{{\underline{x}}{}}
\newcommand{\uy}{{\underline{y}}{}}
\newcommand{\uz}{{\underline{z}}{}}

\newcommand{\uzero}{{\underline{0}}}
\newcommand{\ulambda}{{\underline{\lambda}}}

\newcommand{\uepsilon}{{\underline{\epsilon}}}
\newcommand{\umu}{{\underline{\mu}}}

\newcommand{\utau}{{\underline{\tau}}}

\newcommand{\li}{{\scalebox{0.5}{{$\operatorname{li}$}}}}
\newcommand{\ld}{{\scalebox{0.5}{{$\operatorname{ld}$}}}}

\newcommand{\lii}{{\scalebox{0.5}{$\operatorname{li}$}}}
\newcommand{\ldd}{{\scalebox{0.5}{$\operatorname{ld}$}}}

\newcommand{\tiluw}{{\Tilde{\underline{w}}}}
\newcommand{\tiluy}{{\Tilde{\underline{y}}}}

\newcommand{\uvvec}{{\underline{\Vec{v}}}}

\newcommand{\uxvec}{{\underline{\Vec{x}}}}

\newcommand{\xvec}{{\Vec{x}}}

\newcommand{\Hfour}{{$(\operatorname{H4})$}}

\newcommand{\TKvs}{{T_{K\operatorname{-vs}}}}
\newcommand{\TKvsThe}{{T_{K\operatorname{-vs},\theta}}}
\newcommand{\TKvsTheC}{{T^C_{K\operatorname{-vs},\theta}}}

\newcommand{\RCF}{{\operatorname{RCF}}}

\newcommand{\tp}{{\operatorname{tp}}}

\newcommand{\LK}{{L_K}}

\newcommand{\LKThe}{{L_{K,\theta}}}
\newcommand{\LRC}{{L_{R_C}}}

\newcommand{\Lr}{L_{\operatorname{r}}}

\newcommand{\standartConstruction}{\hyperref[lemma_standart_construction]{Standard Construction}}

\newcommand{\TrafoLemma}{\hyperref[lemma_trafoo]{Transformation Lemma}}

\newcommand{\placeholderNotation}{\hyperref[def_placeholder_notation]{Placeholder-Notation}}

\newtheorem*{notation}{Notation}

\newtheorem{theorem}{Theorem}[section]
\newtheorem*{theorem*}{Theorem}
\newtheorem{theoremi}{Theorem}

\newtheorem{definition}[theorem]{Definition}
\newtheorem{fact}[theorem]{Fact}
\newtheorem{remark}[theorem]{Remark}
\newtheorem{lemma}[theorem]{Lemma}
\newtheorem{corollary}[theorem]{Corollary}
\newtheorem{example}[theorem]{Example}
\newtheorem{observation}[theorem]{Observation}

\newtheorem{subclaim}{Claim}[theorem]

\newtheorem*{subdefinition*}{Definition}
\newtheorem*{subclaim*}{Claim}

\newenvironment{innerproof}[1][Proof]
{%
  \begin{proof}[#1]%
}
{%
  \end{proof}%
}

\usepackage{titlesec}
\titleformat{\section}
  {\normalfont\Large\bfseries\boldmath}{\thesection}{1em}{}
\titleformat{\subsection}
  {\normalfont\large\bfseries\boldmath}{\thesubsection}{1em}{}
\titleformat{\subsubsection}
  {\normalfont\normalsize\bfseries\boldmath}{\thesubsubsection}{1em}{}

\title{Model Theory of Generic Vector Space Endomorphisms II}
\author{Leon Chini}

\newcommand{\Addresses}{{
  \bigskip
  \footnotesize

  \textsc{Mathematisches Institut, Universität Bonn, Endenicher Allee 60, D-53115 Bonn, Germany}\par\nopagebreak
  \textit{E-mail address}: 
  \href{mailto:lchini@uni-bonn.de}{\tt lchini@uni-bonn.de}
}}

\begin{document}
\maketitle

\begin{abstract}
\noindent This paper further studies the model companion of an endomorphism acting on a vector space, possibly with extra structure.  
Given a theory $T$ that $\varnothing$-defines an infinite $K$-vector space $\mathbb{V}$ in every model, we set $T_\theta := T \cup \{\text{``$\theta$ is a $K$-endomorphism of $\mathbb{V}$''}\}$.  
We previously defined a family $\{T^C_\theta : C \in \mathcal{C}\}$ of extensions of $T_\theta$ that parameterizes all consistent extensions of the form  
$$  
    T_\theta \cup \left\{\sum\nolimits_{k}\bigcap\nolimits_{l}\operatorname{Ker}(\rho_{j, k, l}[\theta]) = \sum\nolimits_{k}\bigcap\nolimits_{l} \operatorname{Ker}(\eta_{j, k, l}[\theta]) : j \in \mathcal{J}\right\},  
$$  
where all sums and intersections are finite,  
and all the $\rho[\theta]$'s and $\eta[\theta]$'s are polynomials over $K$ with $\theta$ plugged in.  
Notice that properties such as $\theta^2 - 2\operatorname{Id} = 0$ or ``$\rho[\theta]$ is injective for every $\rho \in K[X] \setminus \{0\}$'' can be expressed in such a manner.  
We also presented a sufficient condition that implies that every $T^C_\theta$ has a model companion $T\theta^C$.  
Under this condition, we characterize all definable sets in $T\theta^C$ and study the completions of $T\theta^C$ as well as the algebraic closure.  
If $T$ is o-minimal and extends $\operatorname{Th}(\mathbb{R}, <)$, we prove that $T\theta^C$ has an o-minimal open core.

%This paper further studies the model companion of an endomorphism acting on a vector space, possibly with extra structure.  
%Given a theory $T$ that $\varnothing$-defines an infinite $K$-vector space $\mathbb{V}$ in every model, we set $T_\theta := T \cup \{\text{``$\theta$ defines a $K$-endomorphism of $\mathbb{V}$''}\}$.  
%We previously defined a family $\{T^C_\theta : C \in \mathcal{C}\}$ of extensions of $T_\theta$ which parametrizes all consistent extensions of the form  
%$$  
%    T_\theta \cup \left\{\sum\nolimits_{k}\bigcap\nolimits_{l}\operatorname{Ker}(\rho_{j, k, l}[\theta]) = \sum\nolimits_{k}\bigcap\nolimits_{l} \operatorname{Ker}(\eta_{j, k, l}[\theta]) : j \in \mathcal{J}\right\},  
%$$  
%where all sums and intersections are finite,  
%and all the $\rho[\theta]$'s and $\eta[\theta]$'s are polynomials over $K$ with $\theta$ plugged in.  
%Notice that properties such as $\theta^2 - 2\Id = 0$ or ``$\rho[\theta]$ is injective for every $\rho \in K[X] \setminus \{0\}$'' can be expressed in such a manner.  
%We also presented a sufficient condition which implies that every $T^C_\theta$ has a model companion $T\theta^C$.  
%Under this condition, we characterize all definable sets in $T\theta^C$ and use this to study the completions of $T\theta^C$, as well as the algebraic closure.  
%If $T$ is o-minimal and extends $\operatorname{Th}(\mathbb{R}, <)$, we prove that $T\theta^C$ has o-minimal open core.
\end{abstract}

\tableofcontents
\section{Introduction}
This paper is a continuation of \cite{Chi25}, which deals with the model companion of an endomorphism acting on a vector space, possibly with extra structure. The goal of this paper is a more fine-grained model-theoretic study of these model companions. For the relevance of this line of inquiry and a description of earlier work, see the introduction of \cite{Chi25}. 

Let $L$ be a language and $T$ a model-complete $L$-theory with an infinite $\varnothing$-definable vector space $\VV$ in every model.  
Given a consistent set $C$ of constraints on an endomorphism, which encodes conditions of the form  
$$
    \sum\nolimits_{k}\bigcap\nolimits_{l}\Ker(\rho_{k, l}[\theta]) = \sum\nolimits_{k}\bigcap\nolimits_{l} \Ker(\eta_{k, l}[\theta])
$$  
(where all sums and intersections are finite,  
and all the $\rho$'s and $\eta$'s are polynomials over $K$), we define the following theory in the language $L_\theta := L \cup \set{\theta}$:  
$$
T^C_\theta := T \cup \set{\text{``$\theta$ is an endomorphism of $\VV$''}} \cup \set{\text{``$\theta$ satisfies the constraints in $C$''}}.
$$  
In \cite{Chi25}, we have shown that $T^C_\theta$ has a model companion $T\theta^C$ if $T$ satisfies a certain condition \Hfour{}, which corresponds to Definition 1.10 in \cite{dEl21b} and basically states that ``$\psi(\ux; \uy)$ implies no finite disjunction of non-trivial linear dependencies in $\ux$ over $\VV$'' can be expressed as an $L$-formula $\sigma_\psi(\uy)$ for every $L$-formula $\psi(\ux; \uy)$.  
These results, as well as all others relevant to this paper, will be recalled in Section \ref{sec_prelim_res}.  

Assuming the condition \Hfour{}, we obtain a description of definable sets. In the following, an algebraic pattern in $\uw$ is essentially a definable function that maps a tuple $\ud$ to a finite subset of $\acl_{L_\theta}(\ud) \cap \VV$:  
\begin{theoremi}[Theorem \ref{theorem_big_fml}] \label{theorem_A}  
    Suppose that $T$ satisfies \Hfour{}. Every $L_\theta$-formula $\phi(\uw)$ is, modulo $T\theta^C$, equivalent to a finite disjunction of formulas of the form  
    $$
    \exists \uy \in Y_\uw : \psi(\uy; \uw)
    $$  
    with $Y$ being an algebraic pattern (see Definition \ref{def_alg_pattern} or the text above) in $\uw$, and $\psi(\uy; \uw)$ being an $L$-formula.  
\end{theoremi}  
\noindent We also obtain a more technical version of Theorem \ref{theorem_A}, which allows us to simplify an $L_\theta$-formula $\phi(\uz; \uw)$ only in the variable $\uw$; see Theorem \ref{theorem_big_fml_preceise}.  

In \cite{Chi25}, we presented a ring $R_C$ of endomorphisms of $\VV$ that are definable in $T\theta^C$ using only the language of the $K$-vector space and $\theta$; see Fact \ref{theorem_r_c_def} for more details.
This ring induces an $R_C$-module structure on $\VV$ that extends its $K$-vector space structure.  
We let $\cl_\theta(A)$ denote the smallest set containing $A$ that is closed under both $\acl_L$ and multiplication by any $r \in R_C$.  
With Theorem \ref{theorem_A}, we obtain the following:  
\begin{theoremi}[Theorem \ref{theorem_ee_iff}]
    Suppose that $T$ satisfies \Hfour{}.
    Two models $(\mm_1, \theta_1)$ and $(\mm_2, \theta_2)$ of $T\theta^C$ are elementarily equivalent if and only if there is an $L_\theta$-isomorphism  
    $$
    \iota \colon (\cl_\theta^{(\mm_1, \theta_1)}(\varnothing), \theta_1) \to (\cl_\theta^{(\mm_2, \theta_2)}(\varnothing), \theta_2)
    $$  
    which is also $L$-elementary with respect to $\mm_1$ and $\mm_2$.  
    If $T$ admits quantifier elimination or if $\acl_L(A) \models T$ for every subset $A \subseteq \mm \models T$, then it suffices for $\iota$ to be an $L_\theta$-isomorphism.  
\end{theoremi}  
\noindent In Corollary \ref{corollary_same_type}, we obtain a similar description for when two tuples have the same type. We also show that $T\theta^C$ is the model completion of  
$$
T^C_\theta \cup \set{\text{``$\theta$ is $C$-image-complete''}}
$$  
(see Corollary \ref{corollary_model_completion}), where $C$-image-completeness (Definition \ref{def_C_image_comple}) can be seen as a ``generalized surjectivity'' condition.  
We are also able to prove a criterion under which $T\theta^C$ has quantifier elimination in a suitable language (see Theorem \ref{theorem_qe}).  
This implies that $\TKvs\theta^C$ has quantifier elimination in the language of $R_C$-modules, where $\TKvs$ is the theory of $K$-vector spaces and $R_C$ is the previously mentioned ring of $\LKThe$-definable endomorphisms.

We also study the algebraic closures in $T\theta^C$ and obtain the following results:
\begin{theoremi}[Theorem \ref{theorem_acl}]
    Suppose that $T$ satisfies \Hfour{}.
    The equation $\acl_{L_\theta}(A) = \cl_\theta(A)$ holds for any set $A \subseteq (\mm, \theta) \models T\theta^C$.
\end{theoremi}
\begin{theoremi}[Theorem \ref{theorem_nes_cond_fixed}] \label{theorem_fake_acl_cond}
    Suppose that $T$ satisfies \Hfour{}. The following conditions are necessary for $\acl_{L_\theta}$ to have the exchange property:
    \begin{enumerate}[(i)]
        \item The algebraic closure $\acl_L$ must have the exchange property in $T$.
        \item The vector space $\VV$ must be one-dimensional with respect to the dimension induced by $\acl_L$ and $C$ must be non-trivial (see Fact \ref{fact_trivvivi}).
        \item The ring $R_C$ must be a field.
    \end{enumerate}
\end{theoremi}
\noindent Note that the conditions in Theorem \ref{theorem_fake_acl_cond} are not sufficient.  
We present $\RCF{}\theta^C$ as an example, where $R_C$ is a field (for the right $C$), $\acl_L$ has the exchange property, but $\acl_{L_\theta}$ does not have the exchange property (see Example \ref{ex_no_ex}).

The work of Block-Gorman in \cite{Blo23} gives many examples of o-minimal theories that satisfy \Hfour{}. We use the technical variant of Theorem \ref{theorem_A} to prove the following:
\begin{theoremi}[Theorem \ref{theorem_o_min_open_core}] \label{theorem_D}
    Let $T$ be an o-minimal expansion of $\Th(\RR,  <)$ which satisfies \Hfour{}.  
    Then $T\theta^C$ has an o-minimal open core.
\end{theoremi}

\noindent \textbf{Acknowledgment.} The author would like to thank Christian d'Elbée for reading an earlier draft of this paper and for providing valuable feedback and suggestions.
The author would also like to thank Philipp Hieronymi for some suggestions and general support.

\section{Preliminary Results} \label{sec_prelim_res}

In this section, we provide an overview of all relevant results from \cite{Chi25}. We omit some highly technical results and cite them when needed.

\subsection{The setting}  
Let $L$ be a first-order language, let $T$ be a model-complete $L$-theory, and let $K$ be a field.
Furthermore, assume that the theory of $K$-vector spaces is definable in $T$.
By this we mean that there are $L$-formulas $\Omega_{\VV}(\ux)$, $\Omega_{0}(\ux)$, $\Omega_{+}(\ux_1, \ux_2, \uy)$, and $\Omega_{\lambda\cdot}(\ux, \uy)$ for each $\lambda \in K$ such that, in every model $\mm \models T$, they define an infinite $K$-vector space $(\VV, 0, +, (\lambda \cdot)_{\lambda \in K})^{\mm}$.
In the case $K = \QQ$, one may also view $(\VV, 0, +)^{\mm}$ as a torsion-free divisible abelian group, since these are precisely the $\QQ$-vector spaces.
If a model $\mm \models T$ is given, then $\VV$ denotes $\VV^\mm$.
If the model is denoted with $\mm'$ instead of $\mm$, then we will write $\VV'$ instead of $\VV$.
If no model of $T$ is clear from the context, then we will still use the letter $\VV$ to denote a $K$-vector space.
We may say ``vector space'' instead of ``$K$-vector space'', ``polynomial'' instead of ``$K$-polynomial'', ``linearly independent'' instead of ``$K$-linearly independent'', and so on. Definiable will always mean $\varnothing$-definable.

\begin{example}
    \label{example_main}
    Two of our main examples are as follows:
    \begin{enumerate}[(i)]
        \item Let $\LK = \set{0, +, (\lambda \cdot)_{\lambda\in K}}$ and let $\TKvs$ be the theory of $K$-vector spaces, with the obvious formulas.
        Then, for any $\mm \models \TKvs$, we choose $(\VV, 0, +, (\lambda \cdot)_{\lambda\in K}) = \mm$.
        Similarly, we can also work with ordered divisible abelian groups, since these are precisely ordered $\QQ$-vector spaces.
        \item Let $K = \QQ$, $\Lr = \set{0, 1, +, \cdot, <}$, and $T = \RCF$, the theory of real closed fields.
        Then, for any $\rr \models \RCF{}$, we choose $(\VV, 0, +, (q \cdot)_{q\in \QQ})^{\rr}$ to be $(\rr_{>0}, 1, \cdot, (x \mapsto x^q)_{q\in \QQ})$.
    \end{enumerate}
\end{example}

\noindent Notationally, we will treat $\VV$ as a unary set, or even as a separate sort.
Any $\LK$-term or formula can then be viewed as an $L$-definable function or an $L$-formula, respectively.
Note that $0$, $+$, and $(\lambda\cdot)_{\lambda\in K}$ need not belong to $L$, so they are not necessarily $L$-terms.
For a given theory, there can be multiple definable vector spaces, as in the case of $\RCF{}$.
Thus, whenever a theory $T$ is given, we actually mean the tuple $(T, \Omega_{\VV}, \Omega_{0}, \Omega_{+}, (\Omega_{\lambda\cdot})_{\lambda \in K})$.

We now define $L_\theta := L(\theta) := L \cup \set{\theta}$, where $\theta$ is a function symbol not contained in $L$.
This has the drawback that, for example, if $\rr \models \RCF$ and the positive elements are viewed as a $\QQ$-vector space, then $\theta$ must also be defined outside $\VV = \rr_{>0}$.
In this case, one might try to set $\theta(0) = 0$ and $\theta(-1) \in \set{1, -1}$ so as to extend $\theta$ to an endomorphism of $(\rr, 1, \cdot)$, but then the ambient structure is no longer a vector space.
For the sake of uniformity, we instead define $\theta(x)$ to be the neutral element of $\VV$ for all $x \not\in \VV$ and set
$$
T_\theta := T \cup \set{\text{``$\theta_{\restriction \VV}$ is a $(\VV, 0, +, (\lambda \cdot)_{\lambda\in K})$-endomorphism''}} \cup \set{\forall x \not\in \VV: \theta(x) = 0}.
$$
In particular, $\theta^n(x) = 0$ for all $x \not\in \VV$ and all $n > 0$.
For consistency, we also define $\theta^0(x) = 0$ whenever $x \not\in \VV$, and $x + y = 0$ whenever either $x \not\in \VV$ or $y \not\in \VV$.
Practically, we will ignore the behavior of $\theta$ outside of $\VV$ and simply treat $\theta$ as a function defined only on $\VV$.
For example, we set $\Ker(\theta) := \set{v \in \VV : \theta(v) = 0}$, and similarly define the kernel for any function that is an endomorphism of $\VV$.

Note that if $\VV$ is an $n$-ary set, then we actually need $n$ different $n$-ary function symbols $\theta_1, \dots, \theta_n$ rather than a single unary function symbol $\theta$.
However, as mentioned above, we will treat elements of $\VV$ as singletons in order to simplify notation and will therefore pretend that $\theta$ is unary.

\begin{definition}
    Given a polynomial $\rho \in K[X]$ and any $d \geq \deg(\rho)$, we let $\rho[\theta]$ denote the endomorphism of $\VV$ defined by
    $
    \rho[\theta](v) := \sum\nolimits_{i=0}^{d} (\rho)_i \cdot \theta^i(v)
    $
    for all $v \in \VV$. Here each $(\rho)_i$ is the respective coefficient of $X^i$ in $\rho$.
\end{definition}

\begin{notation}
    If $\theta$ is clear from the context we may write $\Ker(\rho)$ and $\Image(\rho)$ instead of $\Ker(\rho[\theta])$ and $\Image(\rho[\theta])$.
\end{notation}

\noindent It is easy to check that $\rho[\theta] + \eta[\theta] = (\rho + \eta)[\theta]$ and $\rho[\theta] \circ \eta[\theta] = (\rho \cdot \eta)[\theta]$. Many classical facts for polynomials, such as Bézout's identity or Euclidean division, can easily be translated to this setting. In \cite{Chi25}, we use these results quite a lot, but we do not need them in this paper.

\subsection{Kernel configurations and extensions of $T_\theta$}

\noindent As stated in the introduction, we consider a family $\set{T^C_\theta : C \in \Cc}$ of extensions of $T_\theta$.

\begin{notation}
    We let $\Kp{}$ denote the set of all monic irreducible polynomials over $K$.
\end{notation}

\noindent We start by defining our index set $\Cc$:

\begin{definition} \label{def_kernel_conf}
    We call a pair $(c, d)$ a \textbf{kernel configuration} if
    \begin{enumerate}[(i)]
        \item $c \colon \Kp{} \to \NN \cup \set{\infty}$ is a function; and
        \item $d \in \NN_{> 0} \cup \set{\infty}$ is either $\infty$ or satisfies $d = \sum_{f \in \Kp{}} \deg(f) \cdot c(f)$.
    \end{enumerate}
    We let $\Cc$ denote the set of all kernel configurations.
    Given such a kernel configuration $C = (c, d) \in \Cc$, we set $C(f) := c(f)$ for all $f \in \Kp{}$.
    We define the \textbf{degree} of $C$ by $\deg(C) := d$.
    We say that $C$ is \textbf{algebraic} if $\deg(C) < \infty$.
    In this case, we define the \textbf{minimal polynomial} of $C$ by $\mipo(C) := \prod_{f \in \Kp{}} f^{C(f)}$
    (since $\deg(C) < \infty$, only finitely many factors are different from $1$, so this product is well defined).
    We say that $C$ is \textbf{transcendental} if $\deg(C) = \infty$.
    We let $\Ccalg$ and $\Cctrans$ denote the sets of all algebraic and all transcendental kernel configurations, respectively.
\end{definition}

\noindent Note that the set of kernel configurations depends on the field $K$.
We are now ready to define the family $\set{T^C_\theta : C \in \Cc}$:

\begin{definition} \label{def_T_C_theta}
    Given $C \in \Cc$ and an endomorphism $\theta \colon \VV \to \VV$, we say that $\theta$ is a \textbf{$\mathbf{C}$-endomorphism} if one of the following holds:
    \begin{enumerate}[(i)]
        \item $C$ is algebraic and $\Ker(\mipo(C)) = \VV$, that is, $\mipo(C)[\theta] = 0$.
        \item $C$ is transcendental and $\Ker(f^{C(f)}) = \Ker(f^{C(f)+1})$ for all $f \in \Kp{}$ with $C(f) < \infty$.
    \end{enumerate}
    We define $T^C_\theta := T_\theta \cup \set{\text{``$\theta$ is a $C$-endomorphism''}}$.
\end{definition}

\noindent The family $\set{T^C_\theta : C \in \Cc}$ might seem a bit arbitrary at first; however, notice that any consistent extension of the form
$$
    T_\theta \cup \Set{\sum\nolimits_{k}\bigcap\nolimits_{l}\Ker(\rho_{j, k, l}[\theta]) = \sum\nolimits_{k}\bigcap\nolimits_{l} \Ker(\eta_{j, k, l}[\theta]) : j \in \jj},
$$
where all sums and intersections are finite,  all the $\rho_{j, k, l}$'s and $\eta_{j, k, l}$'s are polynomials over $K$, and $\jj$ is a potentially infinite index set, is equivalent to some $T^C_\theta$ (see Corollary 2.13 in \cite{Chi25} - the proof heavily uses consequences of Bézout's identity). Also note that every $T^C_\theta$ is consistent, and that $T^C_\theta \not\equiv T^{C'}_\theta$ whenever $C \neq C'$ (see Lemma 2.19 in \cite{Chi25}). So the set $\Cc$ parametrizes all consistent extensions as described above, in some sense. Some concrete examples:
\begin{enumerate}[(i)]
    \item Let $C_\infty$ be transcendental with $C_\infty(f) = \infty$ for all $f \in \Kp{}$. One can check that $T_\theta^{C_\infty}$ is $T_\theta$. 
    \item Let $C_0$ be transcendental with $C_0(f) = 0$ for all $f \in \Kp{}$. One can check that $T_\theta^{C_0}$ is $T_\theta \cup \set{\text{``$\rho[\theta]$ is injective''} : \rho \in K[X] \setminus \set{0}}$. In an existentially closed model of $T_\theta^{C_0}$, the maps $\rho[\theta]$ are isomorphisms, so one can solve systems of equations of the form $\bigwedge_{k=1}^m \sum_{l=1}^n \rho_{k, l}[\theta](x_l) = y_k$ just like in a $K(X)$-vector space.
    \item Let $C_f$ be algebraic with $\mipo(C_f) = f \in \Kp{}$. One can, similarly to (ii), solve such systems of equations as in a $K[X]/(f)$-vector space. Here, however, one has the potential advantage that the sequence $\set{\theta^i(v) : i \in \omega}$ is already determined by $\set{\theta^i(v) : 0 \leq i < \deg(f)}$.
\end{enumerate}
These are, in some cases, the ``easiest'' examples to work with, and it can often be helpful to first work with one of these kernel configurations and then turn to the general case. By contrast, the ``hardest'' kernel configurations to work with are those for which $\set{f \in \Kp{} : 0 < C(f) < \infty}$ is infinite.

In \cite{Chi25}, we also showed that every $T^C_\theta$ is inductive (see Lemma 3.2 there). Hence, the model companion of each $T^C_\theta$ exists if and only if the class of existentially closed models of $T^C_\theta$ is elementary. In this case, the model companion is exactly the axiomatization.

\begin{fact} \label{fact_trivvivi}
    If $C$ is \textbf{trivial}, that is, if $\deg(C) = 1$, then the models of $T_\theta^C$ and $T$ are interdefinable.
    Hence, the model companion of $T_\theta^C$ is $T_\theta^C$ itself.
\end{fact}

\noindent We will often implicitly assume that $C$ is non-trivial.

\begin{notation}
    We introduce a few more notations for working with a kernel configuration $C \in \Cc$:
    \begin{enumerate}[(i)]
        \item Given $f \in \Kp{}$ with $C(f) < \infty$, we write $f^C$ instead of $f^{C(f)}$, $f^{C+k}$ instead of $f^{C(f) + k}$, and so on.
        \item Given a finite set $F \subseteq \Kp{}$ with $C(f) < \infty$ for all $f \in F$, we set \hbox{$F^C := \prod_{f \in F} f^C$}.
        \item We define the following subsets of $\Kp{}$:
        $$
        \Kp{C<\infty} := \set{f \in \Kp{} : C(f) < \infty}, \quad \Kp{0<C<\infty} := \set{f \in \Kp{} : 0 < C(f) < \infty},
        $$
        $$
        \Kp{C=0} := \set{f \in \Kp{} : C(f) = 0}, \quad \text{and} \quad \Kp{C=\infty} := \set{f \in \Kp{} : C(f) = \infty}.
        $$
    \end{enumerate}
\end{notation}

\noindent With the above, for any algebraic kernel configuration $C \in \Ccalg$, we have $\deg(C) = \deg(\mipo(C))$ and
    $$
    \mipo(C) = \prod\nolimits_{f \in \Kp{0<C<\infty}} f^C = (\Kp{0<C<\infty})^C.
    $$
\noindent One should note that any $\LKThe$-sentence that holds in $\TKvsTheC$ also holds in $T^C_\theta$ (here $\LKThe$ is the language of $K$-vector spaces with an endomorphism).
To be more precise, one has to modify the sentence accordingly, e.g., quantifiers of the form $\exists x$ must be replaced with $\exists x \in \VV$ and the formulas that define addition/scalar multiplication must be used instead of the function symbols in $\LKThe$.
In general, it also turns out that whenever a model $(\mm, \theta)$ of $T^C_\theta$ is existentially closed, $(\VV, \theta)$ is an existentially closed model of $\TKvsTheC$ (see Remark 3.4 in \cite{Chi25}).

\subsection{$C$-image-completeness}

\noindent Note that the condition of $\theta$ being a $C$-endomorphism does not (at least in the transcendental case) imply any kind of equations that involve the image of $\rho[\theta]$ for some $\rho \in K[X]$.
In Remark 2.21 in \cite{Chi25}, we discussed that it is very unlikely that considering expansions of $T_\theta$ that also impose equations on the images (or mixed equations with sums and intersections of both kernels and images) will lead to new model companions.
However, the following condition holds in any existentially closed model of $T^C_\theta$ and is fundamental to understanding the structure of these models:

\begin{definition} \label{def_C_image_comple}
    We say that an endomorphism $\theta \colon \VV \to \VV$ is \textbf{$\mathbf{C}$-image-complete} if it is a $C$-endomorphism and $\Image(f^{C+1}) = \Image(f^C)$  holds for all $f \in \Kp{C<\infty}$ (recall $\Image(f^C) := \Image(f^{C(f)}[\theta])$).
    We may call a model $(\mm, \theta) \models T_\theta$ \textbf{$\mathbf{C}$-image-complete} if the endomorphism $\theta$ is $C$-image-complete.
\end{definition}

\begin{fact} \label{fact_c_image_comple}
    The following statements hold:
    \begin{enumerate}[(i)]
        \item If $C$ is algebraic, then every $C$-endomorphism is $C$-image-complete (see Lemma 3.11 in \cite{Chi25}).
        \item If $C$ is any kernel configuration and $(\mm, \theta) \models T_\theta^C$ is existentially closed, then $(\mm, \theta)$, or equivalently $\theta$, is also $C$-image-complete (see Corollary 3.13 in \cite{Chi25}).
    \end{enumerate}
\end{fact}

\noindent Note that the converse of (ii) in Fact \ref{fact_c_image_comple} above does not hold.
The most important consequence of $C$-image-completeness is that we can decompose $\VV$ into definable direct summands as follows.

\begin{fact}[Lemma 3.14 in \cite{Chi25}] \label{lemma_decomposition}
    If $\theta \colon \VV \to \VV$ is $C$-image-complete and $F \subseteq \Kp{0<C<\infty}$ is a finite set, then we have
    $$
    \VV = \Image(F^C) \oplus \Ker(F^C) = \Image(F^C) \oplus \bigoplus\nolimits_{f\in F} \Ker(f^C).
    $$
    If $C$ is algebraic and $F = \Kp{0<C<\infty}$, then the summand $\Image(F^C)$ is $\set{0}$ and can therefore be omitted.
\end{fact}

\noindent In \cite{Chi25}, we defined additional endomorphisms of $\VV$ using these decompositions.

\begin{notation}
    For any $\rho \in K[X] \setminus \set{0}$, we let $\Fac(\rho) := \set{f \in \Kp{} : f \mid \rho}$ denote the set of all irreducible factors of $\rho$.
    As a convention, we set $\Fac(0) = \varnothing$.
\end{notation}

\begin{fact}[Lemma 3.15 in \cite{Chi25}] \label{fact_endo_gen}
    In the theory \hbox{$\TKvsThe \cup \set{\text{``$\theta$ is $C$-image-complete''}}$}, the following endomorphisms are definable:
    \begin{enumerate}[(i)]
        \item For $F \subseteq \Kp{0<C<\infty}$ finite, we define the \textbf{projection to the image of $\bm{F^C[\theta]}$} by
        $$
        \pi_{\Image(F^C)}(x) := \text{``the unique $u \!\in\! \Image(F^C)$ for which there is $v \in \Ker(F^C)$ with $x \!=\! u \!+\! v$''.}
        $$
        \item For $F \subseteq \Kp{0<C<\infty}$ finite, we define the \textbf{projection to the kernel of $\bm{F^C[\theta]}$} by
        $$
        \pi_{\Ker(F^C)}(x) := \text{``the unique $v \!\in\! \Ker(F^C)$ for which there is $u \in \Image(F^C)$ with $x \!=\! u \!+\! v$''.}
        $$
        We clearly have $\pi_{\Ker(F^C)} = 1[\theta] - \pi_{\Image(F^C)}$.
        \item For every monic polynomial $\eta \in K[X]$ with $\Fac(\eta) \subseteq \Kp{C<\infty}$, we define the \textbf{pseudo-inverse of $\bm{\eta[\theta]}$} by
        $$
        \eta[\theta]^{-1}(x) := \text{``the unique $u \in \Image(\Fac(\eta)^C)$ with $\eta[\theta](u) = \pi_{\Image(\Fac(\eta)^C)}(x)$''.}
        $$
        Notice that $\Fac(\eta)^C = (\Fac(\eta) \cap \Kp{0<C<\infty})^C$.
        In practice, we will also use $\eta[\theta]^{-1}$ for (non-zero) non-monic polynomials by setting $\eta[\theta]^{-1} := \lambda^{-1} \cdot (\eta/\lambda)[\theta]^{-1}$ for the leading coefficient $\lambda$ of $\eta$.
    \end{enumerate}
\end{fact}

\noindent Notice that whenever $\Fac(\eta) \cap \Kp{0<C<\infty} \neq \varnothing$, we obtain $\eta[\theta] \circ \eta[\theta]^{-1} = \pi_{\Image(\Fac(\eta)^C)} \neq \Id = 1[\theta]$, so the notation might be a bit misleading. Here $\Id$ is the identity on $\VV$.
It is also easy to see that $\Id = 1[\theta] = \pi_{\Image(f^C)} + \pi_{\Ker(f^C)}$. We consider the ring generated by all $\LKThe$-definable endomorphisms we have collected so far:

\begin{fact}[Theorem 3.18 in \cite{Chi25}] \label{theorem_r_c_def}
    We let $R_C$ be the set of all endomorphisms definable in the theory $\TKvsThe \cup \set{\text{``$\theta$ is $C$-image-complete''}}$ that are $\set{+, \circ}$-generated by
    $$
    \set{\rho[\theta] : \rho \in K[X]} \cup \set{\pi_{\Image(F^C)} : F \subseteq \Kp{0<C<\infty} \text{ finite}} \cup \set{\eta[\theta]^{-1} : \eta \text{ monic with }\Fac(\eta) \subseteq \Kp{C<\infty}}.
    $$
    The structure $(R_C, 0[\theta], 1[\theta], +, \circ)$ is a unitary commutative ring with $\Char(R_C) = \Char(K)$. In fact, it can also be seen as $K$-algebra, as $K \subseteq R_C$ (identifying $\lambda \in K$ with the endomorphism $x \mapsto \lambda \cdot x$ which is $\lambda[\theta]$).
\end{fact}

\noindent We may sometimes write $0$ instead of $0[\theta]$ and $\Id$ or $1$ instead of $1[\theta]$. In particular, if we regard $R_C$ purely as a ring, we may write $(R_C, 0, 1, +, \cdot)$. The ring $R_C$ may again look complicated at first, but for the ``easiest to work with'' kernel configurations, we obtain the following:
\begin{enumerate}[(i)]
    \item $(R_{C_\infty}, 0, 1, +, \cdot) \simeq (K[X], 0, 1, +, \cdot)$ for the unique kernel configuration $C_\infty \in \Cctrans$, which satisfies \hbox{$C_\infty(f) = \infty$} for all $f \in \Kp{}$.
    \item $(R_{C_0}, 0, 1, +, \cdot) \simeq (K(X), 0, 1, + , \cdot)$ for the unique kernel configuration $C_0 \in \Cctrans$ with $C_0(f) = 0$ for all $f \in \Kp{}$. 
    Notice that this is a field.
    \item $(R_{C}, 0, 1, +, \cdot) \simeq (K[X]/(\mipo(C)), 0, 1, +, \cdot)$ for all $C \in \Ccalg$.
    This also implies that our ring $(R_{C}, 0, 1, +, \cdot)$ is a field for all algebraic kernel configurations $C$ with $\mipo(C)$ being irreducible.
\end{enumerate}
Other examples of $R_C$ can be found in Corollary 3.25 in \cite{Chi25}. 
There, the case where $C$ is transcendental and $\Kp{0<C<\infty}$ is infinite again turns out to be the most complicated. Multiplication rules, such as $\rho[\theta] \circ \pi_{\Image(F^C)} = \rho[\theta]$ if $F^C \mid \rho$, can be found in Lemma 3.21 in \cite{Chi25}.  

\begin{fact}[see Remark 3.26 in \cite{Chi25}] \label{fact_when_field}
    $R_C$ is a field if and only if $C = C_0$, as in (ii) above, or if $C$ is algebraic with $\mipo(C)$ irreducible. 
\end{fact}

\noindent Also note that the elements of $R_C$ are, as defined in Fact \ref{theorem_r_c_def}, definable functions in the theory $\TKvsThe \cup \set{\text{``$\theta$ is $C$-image-complete''}}$.
By this, we mean that the elements of $R_C$ are equivalence classes of $\LKThe$-formulas modulo the theory \hbox{$\TKvsThe \cup \set{\text{``$\theta$ is $C$-image-complete''}}$} that define an endomorphism in every model of $\TKvsThe \cup \set{\text{``$\theta$ is $C$-image-complete''}}$.
So, in order to prove $r = r'$, we need to show
$$
    r^{(\VV, \theta)} = r'^{(\VV, \theta)} \quad \text{for all $(\VV, \theta) \models \TKvsThe \cup \set{\text{``$\theta$ is $C$-image-complete''}}$},
$$
and, in order to prove $r \neq r'$, we need to find $(\VV, \theta) \models \TKvsThe \cup \set{\text{``$\theta$ is $C$-image-complete''}}$ with
$
    r^{(\VV, \theta)} \neq r'^{(\VV, \theta)}.
$
In Remark 3.20 in \cite{Chi25}, we showed that $r \neq r'$ implies $r^{(\VV, \theta)} \neq r'^{(\VV, \theta)}$ if $(\VV, \theta)$ is an existentially closed model of $\TKvsTheC$. 
In a future paper, we will see that in an existentially closed model of $T_\theta^C$, any $\LKThe$-definable endomorphism of $\VV$ is, in fact, in $R_C$.

\subsection{Existentially closed models and first-order axiomatization}

We now state the characterization the existentially closed models of $T^C_\theta$ from \cite{Chi25}. We start with the remaining ingredients:

\begin{definition}
    \label{def_param_c_sequence_system} A \textbf{parametrized $\mathbf{C}$-sequence-system} is an $\LKThe$-formula of the form
    $$
    S(\ux; \uy) = \bigwedge\nolimits_{k=1}^n f_k^{q_k}[\theta](x_{\ldd, k}) = y_k
    $$
    with $\ux := \ux_\li\ux_\ld := (x_{\lii, k} : 1 \leq k \leq m)(x_{\ldd, k} : 1 \leq k \leq n)$ and $\uy = (y_1, \dots, y_n)$ that satisfies the following conditions:
    \begin{enumerate}[(i)]
        \item If $C$ is algebraic, then $m = 0$.
        \item $f_k \in \Kp{0<C}$ and $q_k \in \set{q \in \NN : 0 < q \leq C(f_k)}$ hold for all $k \in \set{1, \dots, n}$.
    \end{enumerate}
    
\end{definition}
\noindent We will always denote parametrized $C$-sequence-systems by the letter $S$.
Given such a parametrized $C$-sequence-system $S(\ux; \uy)$, we assume that everything is as above, that is, $m$, $n$, and the $f_k$'s and $q_k$'s are defined implicitly, and we set $\ux = \ux_\li\ux_\ld$ as above.
When we partition $\ux = \ux_\li\ux_\ld$ as above, we think of:
    \begin{enumerate}[(i)]
        \item $\ux_\li$ as the linearly independent part of $\ux$, since $S(\ux; \uy)$ does not imply any linear dependencies for the sequence $(\theta^i(x_{\lii, k}) : 1 \leq k \leq m, i \in \omega)$;
        \item $\ux_\ld$ as the linearly dependent part of $\ux$, since the formula $S(\ux; \uy)$ implies that the sequence $(\theta^i(x_{\ldd, k}) : i \in \omega)$ is linearly dependent over $\spanA{y_k}{K}$ for each $k \in \set{1, \dots, n}$.
    \end{enumerate}
The names $\ux_\li$ and $\ux_\ld$ for these tuples are abbreviations for linearly independent and linearly dependent. Notice that in the algebraic case, we require $\ux_\li$ to be empty, which makes sense, as we have $\sum_{i=0}^{\deg(\mipo(C))} (\mipo(C))_i \cdot \theta^i(v) = 0$ for any $v \in \VV$ in that case.

\begin{definition}
    \label{def_c_sequence_system} \label{def_compatible} Let $S(\ux; \uy) = \bigwedge_{k=1}^n f_k^{q_k}[\theta](x_{\ldd, k}) = y_k$ be a parametrized $C$-sequence-system as in Definition \ref{def_param_c_sequence_system}, and let $(\VV, \theta) \models \TKvsThe$ be given.
\begin{enumerate}[(i)]
    \item We say that a tuple $\uu = (u_1, \dots, u_n) \in \VV$ is \textbf{compatible} with $S$ if $u_k \in \Ker(f_k^{C-q_k})$ for every $k \in \set{1, \dots, n}$ with $f_k \in \Kp{0<C<\infty}$.
    \item A \textbf{$\mathbf{C}$-sequence-system} over $(\VV, \theta)$ is an $\LKThe(\VV)$-formula of the form
    \hbox{$
    S(\ux) = S'(\ux; \uu)
    $}
    where $S'$ is a parametrized $C$-sequence-system and $\uu \in \VV$ is compatible with $S'$.
\end{enumerate}
\end{definition}

\noindent We will also denote $C$-sequence-systems over some $(\VV, \theta) \models \TKvsTheC$ by the letter $S$.
Notice that a $C$-sequence-system over $(\VV, \theta)$ is also a $C$-sequence-system over any extension $(\VV', \theta')$ that is also a model of $\TKvsThe$.

\begin{definition}[Placeholder notation] \label{def_placeholder_notation}
    Let $\ux = (x_k : k \in \kk)$ be a tuple of variables.
    We define the \textbf{placeholder sequence} \hbox{$\uxvec := (x^i_k : k \in \kk, i \in \omega)$} to be a new tuple of variables.
    We call each $x^i_k$ a \textbf{placeholder variable} or a \textbf{placeholder} for $\theta^i(x_k)$.
    We furthermore define:
    \begin{enumerate}[(i)]
        \item $\ux^i := (x_k^i : k \in \kk)$ for each $i \in \omega$, and
        \item $\xvec_k := (x^i_k : i \in \omega)$ for each $k \in \kk$.
    \end{enumerate}
    We may sometimes write $(\ux^i : i \in \omega)$ or $(\xvec_k : k \in \kk)$ instead of $\uxvec$.
    For a singleton $x$, we similarly define $\xvec := (x^i : i \in \omega)$.
    If a formula $\psi(\uxvec; \uw)$ is given, we define:
    \begin{enumerate}[(i)]
        \setcounter{enumi}{2}
        \item $\psi_\theta(\ux; \uw) := \psi((\theta^i(x_k) : k \in \kk, i \in \omega); \uw)$.
    \end{enumerate}
\end{definition}

\begin{definition} \label{def_formual_bounded}
    Let $S(\ux; \uy)= \bigwedge\nolimits_{k=1}^n f_k^{q_k}[\theta](x_{\ldd, k}) = y_k$ be a parametrized $C$-sequence-system as in Definition \ref{def_param_c_sequence_system}.
    If $\psi(\uxvec; \uw)$ is a formula, then we say $\psi(\uxvec; \uw)$ is \textbf{bounded} by $S$ if one of the following equivalent conditions holds:
        \begin{enumerate}[(i)]
            \item For all $k \in \set{1, \dots, n}$, the variable $x^i_{\ldd, k}$ does not appear in $\psi(\uxvec; \uw)$ for \hbox{$i \geq \deg(f_k^{q_k})$}.
            \item For all $k \in \set{1, \dots, n}$, the term $\theta^i(x_{\ldd, k})$ does not appear in $\psi_\theta(\ux; \uw)$ for \hbox{$i \geq \deg(f_k^{q_k})$}.
        \end{enumerate}
    We also say that $\psi_\theta(\ux; \uw)$ is \textbf{bounded} by $S$, if $\psi(\uxvec; \uw)$ is bounded by $S$. Furthermore, we say that a formula is \textbf{bounded} by a $C$-sequence-system $S$ (i.e., a parametrized $C$-sequence-system with some compatible parameters plugged in) if it is bounded by the underlying parametrized $C$-sequence-system.
\end{definition}

\noindent In practice, for a formula $\psi(\uxvec)$ to be bounded by $S$ means that no subterm of the form $\theta^i(x_{\ldd, k})$ appearing in $\psi_\theta(\ux)$ can be replaced by applying a Euclidean division with the equation $f_k^{q_k}[\theta](x_{\ldd, k}) = y_k$ from $S(\ux; \uy)$. Indeed, if $i \geq \deg(f^{q_k}_k)$, then we could replace $\theta^i(x_{\ldd, k})$ with $\chi[\theta](y_k) + r[\theta](x_{\ldd, k})$, where $\chi, r \in K[X]$ are the unique polynomials  satisfying $\chi \cdot f^{q_k}_k + r = X^i$ and $\deg(r) < \deg(f^{q_k}_k) \leq i$.

Now that we have all ingredients, we can state the characterization of existentially closed models of $T^C_\theta$:

\begin{theorem}[Theorem 3.34 in \cite{Chi25}] \label{theorem_big_characterization}
    $(\mm, \theta) \models T^C_\theta$ is existentially closed if and only if it is $C$-image-complete and 
    $$
        (\mm, \theta) \models \exists \ux \in \VV : \psi_\theta(\ux) \wedge S(\ux)
    $$
    holds for any $C$-sequence-system $S(\ux)$ over $(\VV, \theta)$ and $L(M)$-formula $\psi(\uxvec)$ that is bounded by $S$ and does not imply any finite disjunction of non-trivial linear dependencies in $\uxvec$ over $\VV$.
\end{theorem}

\noindent We give two examples where the characterization simplifies quite a lot:
\begin{enumerate}[(i)]
    \item Fix $f \in \Kp{}$ and let $C_f$ be the unique algebraic kernel configuration with $\mipo(C_f) = f$.
    A model $(\mm, \theta) \models T^{C_f}_\theta$ is existentially closed if and only if
    $$
        (\mm, \theta) \models \exists \ux \in \VV : \psi(\theta^0(\ux), \dots, \theta^{\deg(f)-1}(\ux))
    $$
    holds for every $L(M)$-formula $\psi(\ux^0, \dots, \ux^{\deg(f)-1})$ that does not imply any finite disjunction of non-trivial linear dependencies in $\ux^0, \dots, \ux^{\deg(f)-1}$ over $\VV$.
    This is Theorem 3.33 in \cite{Chi25}.
    
    \item Let $C_0$ be the unique transcendental kernel configuration with $C_0(f) = 0$ for all $f \in \Kp{}$.
    A model $(\mm, \theta) \models T^{C_0}_\theta$ is existentially closed if and only if $\rho[\theta]$ is invertible for every $\rho \in K[X] \setminus \set{0}$ and
    $$
        (\mm, \theta) \models \exists \ux \in \VV : \psi_\theta(\ux)
    $$
    holds for every $L(M)$-formula $\psi(\uxvec)$ that does not imply any finite disjunction of non-trivial linear dependencies in $\uxvec$ over $\VV$.
    This is Theorem 3.34 in \cite{Chi25}.
\end{enumerate}

\noindent The next step is to first-order axiomatize this characterization when possible.
For this, we need the following two families of formulas:

\begin{fact}[Lemma 3.36 in \cite{Chi25}]
    Given a parametrized $C$-sequence-system $S(\ux; \uy)$, there is a $\LKThe$-formula $\delta_S(\uy)$ such that $(\mm, \theta) \models \delta_S(\uu)$ holds if and only if $\uu$ is compatible with $S$.
\end{fact}

\begin{definition}[see Definition 1.11 in \cite{dEl21b}] \label{def_hfour}
    We say that $T$ (with the specific choice of $\VV$) satisfies $(\operatorname{H4})$ if, for every $L$-formula $\psi(\ux; \uw)$, there is an $L$-formula $\sigma_\psi(\uw)$ such that, for all $\mm \models T$ and $\ud \in M$, we have $\mm \models \sigma_\psi(\ud)$ if and only if one of the following two equivalent conditions holds:
    \begin{enumerate}[(i)]
        \item The formula $\psi(\ux; \ud)$ implies no finite disjunction of non-trivial linear dependencies in $\ux$ over $\VV$.
        \item There are an elementary extension $\mm' \succ \mm$ and a tuple $\uv' \in \VV'$ linearly independent over $\VV$ such that $\mm' \models \psi(\uv'; \ud)$.
    \end{enumerate}
\end{definition}

\begin{theorem}[Theorem 3.39 in \cite{Chi25}] \label{theorem_first_oder}
If $T$ satisfies \Hfour{}, then $T_\theta^C$ has a model companion $T\theta^C$, i.e., a first-order axiomatization of the class of existentially closed models.
It is axiomatized by the theory \hbox{$T_\theta \cup \set{\text{``$\theta$ is $C$-image-complete''}}$} together with the sentence
$$
\forall \uw:\forall \uy \in \VV : (\sigma_\psi(\uw) \wedge \delta_S(\uy)) \rightarrow \exists \ux \in \VV : \psi_\theta(\ux; \uw) \wedge S(\ux; \uy)
$$
for every parametrized $C$-sequence-system $S(\ux; \uy)$ and every $L$-formula $\psi(\uxvec; \uw)$ that is boun\-ded by $S$.
\end{theorem}

\begin{example} \label{examples_hfour}
    \Hfour{} holds in the following settings:
    \begin{enumerate}[(i)]
        \item The theory $\TKvs$ with the vector space $(\VV, +, 0, (\lambda \cdot)_{\lambda \in K})$ being the entire structure satisfies \Hfour{}.
        This follows easily from quantifier elimination.
        \item Any complete and model-complete o-minimal theory $T$ extending the theory of divisible ordered abelian groups, with $(\VV, +, 0, (q \cdot)_{q \in \QQ})$ being a subinterval of the line (but not necessarily a subgroup) and continuous operations, satisfies \Hfour{} if and only if there is no infinite definable family of germs of $(\VV, +, 0, (q \cdot)_{q \in \QQ})$-endomorphisms at $0_\VV$; combine Theorem 2.4 and Lemma 2.9 in \cite{Blo23}.
        \item Any complete and model-complete o-minimal expansion of $\RCF{}$ with $(\VV, +, 0, (q \cdot)_{q \in \QQ})$ given by $(R_{>0}, \cdot, 1, (x \mapsto x^q)_{q\in \QQ})$ satisfies \Hfour{} if and only if no partial exponential function is definable.
        This is a special case of (ii); see the proof of Theorem A in \cite{Blo23}.
        \item Let $\FF_q$ be a finite field with $q = p^r$.
        If $T$ expands the theory of an $\FF_q$-vector space and $\VV$ is that vector space, then $T$ satisfies \Hfour{} if and only if it eliminates $\exists^\infty$; see the proof of Theorem 5.2 in \cite{dEl21b}.

        For expansions of the theories $\ACF{}_p$, $\operatorname{SCF}_{p, e}$ ($e$ either finite or infinite), $\operatorname{Psf}_p$, $\operatorname{ACFA}_p$, and $\operatorname{DCF}_p$, this implies that whenever $\FF_q$ is contained in every model as constants, \Hfour{} holds with the $\FF_q$-vector space given by addition.
        For more details, see Example 5.10 in \cite{dEl21b}.
    \end{enumerate}
\end{example}

\noindent We give one final definition for (parametrized) $C$-sequence-system:

\begin{definition} \label{def_rk_deg_def}
    Let $S(\ux; \uy) = \bigwedge_{k=1}^n f_k^{q_k}[\theta](x_{\ldd, k}) = y_k$ be a parametrized $C$-sequence-system, as in Definition \ref{def_param_c_sequence_system} (i.e., $\ux = \ux_\li\ux_\ld$ and so on).
    We define:
    \begin{enumerate}[(i)]
        \item The \textbf{rank} of $S(\ux; \uy)$ as $\rk(S) := m = |\ux_\li|$.
        \item The \textbf{degree} of $S(\ux; \uy)$ as $\deg(S) := \sum\nolimits_{k=1}^n \deg(f_k^{q_k})$.
    \end{enumerate}
    We define the rank and degree of a $C$-sequence-system (i.e., a parametrized $C$-sequence-system with compatible parameters plugged in) to be the rank and degree of the underlying parametrized $C$-sequence-system.
\end{definition}

\noindent So, for every (parametrized) $C$-sequence-system, $(\rk(S), \deg(S))$ lies in $\NN^2$.  
Note that $(\NN^2, <_\Lex)$ is a well-ordered set, where $<_\Lex$ is the lexicographical order, which defines the first entry to be more significant than the second entry.
This allows us to do inductions on $C$-sequence-systems.

\section{Definable Sets}

We now assume that $T$ satisfies \Hfour{}, which guarantees the existence of $T\theta^C$, the model companion of $T^C_\theta$, for all kernel configurations $C \in \Cc$ by Theorem \ref{theorem_first_oder}.
The goal of this section is to describe definable sets in $T\theta^C$ as clearly as possible.

\begin{definition} \label{def_alg_pattern}
    An \textbf{algebraic pattern} in $\uw$ is an $L_\theta$-formula of the form
    $$
    \psi(y_1, \dots, y_m; \uw) := \bigwedge\nolimits_{k=1}^m \exists x \!\in\! \VV : y_k = r_k(x) \wedge \varphi_k(x; y_1, \dots, y_{k-1}; \uw),
    $$
    where each $r_k$ belongs to the ring $R_C$ (see Fact \ref{theorem_r_c_def}), and each $\varphi_k(x; y_1, \dots, y_{k-1}; \uw)$ is an $L$-formula algebraic in $x$.
    In practice, we identify an algebraic pattern with the set $Y$ it defines.
    In this case, we let $\psi_Y(y_1, \dots, y_m; \uw)$ denote the corresponding formula.
    Given a $|\uw|$-sized tuple $\ud$ in a model $(\mm, \theta)$ of $T_\theta \cup \set{\text{``$\theta$ is $C$-image-complete''}}$, we let $Y_\ud$ denote the set defined by $\psi_Y(y_1, \dots, y_m; \ud)$.
\end{definition}

\noindent With the definition above, we can already state the main theorem of this chapter:

\begin{theorem} \label{theorem_big_fml}
    Assume that $T$ satisfies \Hfour{}. Then every $L_\theta$-formula $\phi(\uw)$ is, modulo $T\theta^C$, equivalent to a finite disjunction of formulas of the form
    $$
    \exists \uy \in Y_\uw : \psi(\uy; \uw),
    $$
    where $Y$ is an algebraic pattern in $\uw$, and $\psi$ is an $L$-formula.
\end{theorem}

\noindent In Section \ref{sec_completions_etc}, we will use the theorem above to study the completions of the model companion $T\theta^C$, as well as some quantifier elimination results; in this section, we focus only on the proof.

\subsection{Simplification up to a sub-tuple}

In a future paper, we will encounter situations where an $L_\theta$-formula $\phi(\uz; \uw)$ is given and we only want to ``simplify'' the formula in the tuple $\uw$.
We present the corresponding theorem (Theorem \ref{theorem_big_fml_preceise}) in this section and show how to prove Theorem \ref{theorem_big_fml} using it.

For some fixed theories, all definable sets are more or less built from special formulas (e.g., cells in o-minimal theories).
The following definition allows us to refine the set of $L(M)$-formulas $\psi(\uxvec)$ that we need to consider in our characterization of existentially closed models of $T\theta^C$ (see Remark \ref{rem_refine_char}):

\begin{definition} \label{def_suitable_set}
    We call a set $\dd$ of $L$-formulas $\psi(\uxvec; \uz; \uw)$, where $|\ux|$, $|\uz|$, and $|\uw|$ are finite but not fixed, \textbf{suitable} if, for any formula $\psi(\uxvec; \uz; \uw) \in \dd$, any $\mm \models T$, and any $\ud \in M$, the $L(M)$-formula $\psi(\uxvec; \uz; \ud)$ is either inconsistent or implies no finite disjunction of non-trivial linear dependencies in $\uxvec$ over $\VV$.

    We say that a suitable set $\dd$ has a \textbf{splitting strategy} if, for any $L$-formula $\psi(\uxvec; \uz; \uw)$, there is a finite set of $L$-formulas $\set{\psi_1(\uxvec; \uz; \uw), \dots, \psi_m(\uxvec; \uz; \uw)}$ such that
    $$
        T \models \forall \uz\uw : \forall \uxvec \in \VV : \psi(\uxvec; \uz; \uw) \leftrightarrow \bigvee\nolimits_{k=1}^m \psi_k(\uxvec; \uz; \uw)
    $$
    holds, and, for any $k \in \set{1, \dots, m}$, the variables that appear in $\psi_k(\uxvec; \uz; \uw)$ form a subset of the variables that appear in $\psi(\uxvec; \uz; \uw)$, and either:
    \begin{enumerate}[(i)]
        \item $\psi_k(\uxvec; \uz; \uw)$ lies in $\dd$, or
        \item there is an $L$-formula $\varphi(y; \uw)$ that defines finite subsets of $\VV$ in the variable $y$, and a tuple of polynomials $\ulambda = (\lambda_{1}, \dots, \lambda_{|\ux|}) \in K[X]^{|\ux|} \setminus \set{\uzero}$ (both $\varphi(y; \uw)$ and $\ulambda$ may depend on $k$), such that the following holds:
        $$
            T \models \forall \uz\uw : \forall \uxvec \in \VV : \psi_k(\uxvec; \uz; \uw) \rightarrow \varphi\Big(\! \sum\nolimits_{l=1}^{|\ux|}\sum\nolimits_{i=0}^{\deg(\lambda_{l})} (\lambda_{l})_i \cdot x^i_l; \uw \Big).
        $$
    \end{enumerate}
\end{definition}

\noindent In most cases, it will be sufficient to choose $\dd$ as the following suitable set, which also admits a splitting strategy if \Hfour{} holds, as will be shown in Lemma \ref{lemma_Hfour_splitting_strategy}:
$$
\Set{ \psi(\uxvec; \uz; \uw) \in L : \parbox{11.1cm}{``for all $\ud \in \mm \models T$, the formula $\psi(\uxvec; \uz; \ud)$ is either inconsistent or implies\\ ${}$\hspace{2pt} no finite disjunction of non-trivial linear dependencies in $\uxvec$ over $\VV$''}}.
$$
In Section \ref{sec_o_min_open_core}, we also present a modified set $\dd_o$ for the case where $T$ is an o-minimal theory.
This set $\dd_o$ refines the set above so that the condition ``$\psi(\uxvec; \uz; \ud)$ is either inconsistent or implies no finite disjunction of non-trivial linear dependencies in $\uxvec$ over $\VV$'' holds locally with respect to the topology generated by the order.
We then use $\dd_o$ in combination with Theorem \ref{theorem_big_fml_preceise} below to prove that $T\theta^C$ has an o-minimal open core.

\begin{definition} \label{def_alg_pat_comp}
    Let $Y$ be an algebraic pattern in $\uw$ defined by $\psi_Y(\uy; \uw)$, let $\pi$ be a coordinate projection, and let $S(\ux; \tiluy)$ be a parametrized $C$-sequence-system.
    We say that the pair $(Y, \pi)$ is \textbf{compatible} with $S$ if
    $$
    T_\theta \cup \set{\text{``$\theta$ is $C$-image-complete''}} \models \forall\uw : \forall \uy \in Y_\uw : \text{``$\pi(\uy)$ is compatible with $S$''}.
    $$
\end{definition}

\noindent We are now ready to state the version of Theorem \ref{theorem_big_fml} that only simplifies $L_\theta$-formulas $\phi(\uz;\uw)$ in the tuple $\uw$.

\begin{theorem} \label{theorem_big_fml_preceise}
    \newcounter{savedtheorem55}
    \setcounter{savedtheorem55}{\value{theorem}}
    Suppose $\dd$ is a suitable set of $L$-formulas that has a splitting strategy.
    Any existential $L_\theta$-formula $\phi(\uz;\uw)$ is, modulo $T_\theta \cup \set{\text{``$\theta$ is $C$-image-complete''}}$, equivalent to a finite disjunction of formulas of the form
    $$
    \exists \uy \in Y_\uw : \exists \ux \in \VV : \psi_\theta(\ux; \uz; \uy\uw) \wedge S(\ux; \pi(\uy)), \quad \text{where}
    $$
    \begin{enumerate}[(i)]
        \item $S(\ux; \tiluy)$ is a parametrized $C$-sequence-system;
        \item $\psi(\uxvec; \uz; \uy\uw) \in \dd$ is bounded by $S$;
        \item $Y$ is an algebraic pattern in $\uw$, $\pi$ is a projection, and $(Y, \pi)$ is compatible with $S$.
    \end{enumerate}
\end{theorem}

\noindent Arguably, the theorem above looks nothing like Theorem \ref{theorem_big_fml}.
However, by letting $\uz$ be empty and applying the axioms of $T\theta^C$, we can almost prove Theorem \ref{theorem_big_fml}:

\begin{proof}[Proof of Theorem \ref{theorem_big_fml}]
    Note that this proof assumes that there is a suitable set $\dd$ with a splitting strategy.
    This will be shown in Lemma \ref{lemma_Hfour_splitting_strategy}.

    Since $T\theta^C$ is model-complete, every $L_\theta$-formula is equivalent to an existential one.
    Hence, by Theorem \ref{theorem_big_fml_preceise}, any formula $\phi(\uw)$ is, modulo $T\theta^C$, equivalent to a finite disjunction of formulas of the form
    $$
    \exists \uy \in Y_\uw : \exists \ux \in \VV : \psi_\theta(\ux; \uy\uw) \wedge S(\ux; \pi(\uy)), \quad \text{where}
    $$
    \begin{enumerate}[(i)]
        \item $S(\ux; \uy)$ is a parametrized $C$-sequence-system;
        \item $\psi(\uxvec; \uy\uw) \in \dd$ is bounded by $S$.
        Here, $\psi(\uxvec; \uy\uw) \in \dd$ means that for any $\mm \models T$ and tuples $\uu \in \VV$ and $\ud \in M$, the formula $\psi(\uxvec; \uu\ud)$ is either inconsistent or implies no finite disjunction of non-trivial linear dependencies in $\uxvec$ over $\VV$;
        \item $Y$ is an algebraic pattern in $\uw$, $\pi$ is a projection, and $(Y, \pi)$ is compatible with $S$.
        This means that for any $\ud \in (\mm, \theta) \models T_\theta \cup \set{\text{``$\theta$ is $C$-image-complete''}}$ and $\uu \in Y_\ud$, the tuple $\pi(\uu)$ is compatible with $S$, i.e., $S(\ux; \pi(\uu))$ is a $C$-sequence-system.
    \end{enumerate}
    With our characterization of existentially closed models of $T^C_\theta$ (see Theorem \ref{theorem_big_characterization}), we immediately see that any formula as above is, modulo $T\theta^C$, equivalent to
    $$
    \exists \uy \in Y_\uw : \exists \uxvec \in \VV : \psi(\uxvec; \uy\uw).
    $$
    Since $\exists \uxvec \in \VV : \psi(\uxvec; \uy\uw)$ is an $L$-formula, we conclude.
\end{proof}

\begin{remark} \label{rem_refine_char}
    Let $\dd$ be a suitable set with a splitting strategy.
    In Theorem \ref{theorem_big_characterization}, it suffices to consider $L(M)$-formulas $\psi(\uxvec)$ of the form $\psi'(\uxvec; \ud)$ for some $\psi'(\uxvec;\uw) \in \dd$ (with $\uz$ empty).
    Similarly, one can improve Theorem \ref{theorem_first_oder}.
\end{remark}

\subsection{A suitable set with a splitting strategy}

In order to complete the proof of Theorem \ref{theorem_big_fml}, we still need to prove Theorem \ref{theorem_big_fml_preceise} and show that \Hfour{} implies the existence of a suitable set of formulas with a splitting strategy.
In this section, we deal with the latter.

\begin{lemma} \label{lemma_hfour_fml}
    If $T$ satisfies \Hfour{}, then one can choose the formula $\sigma_\psi(\uw)$ (see Definition \ref{def_hfour}) for each $L$-formula $\psi(\ux; \uw)$ with $\ux = (x_1, \dots, x_n)$ to be of the form
    $$
    \exists \ux \in \VV : \psi(\ux; \uw) \wedge \bigwedge\nolimits_{k=1}^m \neg\varphi_k\Big(\sum\nolimits_{l=1}^n \lambda_{k, l} \cdot x_l; \uw\Big),
    $$
    with $(\lambda_{k, 1}, \dots, \lambda_{k, n}) \in K^n \setminus \set{\uzero}$ and $\varphi_k(y; \uw)$ algebraic in $y$ for each $k \in \set{1, \dots, m}$.
\begin{proof}
    Clearly, $\sigma_\psi(\uw)$ implies every formula of this form.
    We show that some formula of this form implies $\sigma_\psi(\uw)$.
    Suppose, toward a contradiction, that no such formula implies $\sigma_\psi(\uw)$.
    By compactness, we can find a model $\mm \models T$ and $\ud \in M$ such that $\mm \models \neg\sigma_\psi(\ud)$ and
    \begin{align}
        \mm \models \exists \ux \in \VV : \psi(\ux; \ud) \wedge \bigwedge\nolimits_{k=1}^m \neg\varphi_k\Big(\sum\nolimits_{l=1}^n \lambda_{k, l} \cdot x_l; \ud\Big) \label{tag_fml_sigma_wtf}
    \end{align}
    holds for every formula as described in the statement.
    As $\mm \models \neg\sigma_\psi(\ud)$, the partial type
    $$
    \set{\psi(\ux; \ud)} \cup \set{\text{``$\ux$ is linearly independent over $\VV$''}}
    $$
    is inconsistent.
    Thus we have
    $
    \mm \models \forall \ux \in \VV : \psi(\ux; \ud) \rightarrow \bigvee\nolimits_{k=1}^m \big(\sum\nolimits_{l=1}^{n} \lambda_{k, l} \cdot x_l\big) \in A_k
    $
    for some finite sets $A_k \subseteq \VV$ and tuples $\ulambda_k = (\lambda_{k, 1}, \dots, \lambda_{k, n}) \in K^n \setminus \set{\uzero}$.
    Because of (\ref{tag_fml_sigma_wtf}), we can choose each $A_k$ such that $A_k \cap \acl_L(\ud) = \varnothing$.
    By a standard argument, we can find $A'_1, \dots, A'_m \equiv_{\acl_L(\ud)} A_1, \dots, A_m$ in some elementary extension $\mm' \succ \mm$ such that $A'_k \cap M = \varnothing$ for all $k$; use, e.g., (1) of Proposition 1.5 in \cite{Adl09}.
    We obtain
    $$
    \mm' \models \forall \ux \in \VV : \psi(\ux; \ud) \rightarrow \bigvee\nolimits_{k=1}^m \big(\sum\nolimits_{l=1}^{n} \lambda_{k, l} \cdot x_l\big) \in A'_k.
    $$
    Now, for any $\uv \in \VV$, we cannot have $\mm' \models \psi(\uv; \ud)$, since $\sum\nolimits_{l=1}^{n} \lambda_{k, l} \cdot v_l \in \VV \subseteq M$ and $A'_k \cap M = \varnothing$.
    Since $\mm \prec \mm'$, we see that $\psi(\ux; \ud)$ is inconsistent.
    But by our definition of $\ud$, we know that $\exists \ux \in \VV : \psi(\ux; \ud)$ holds, a contradiction.
\end{proof}
\end{lemma}

\noindent In the following lemma, we have $\uxvec = (x^i_k : 1 \leq k \leq n, i \in \omega)$ as in our \placeholderNotation{}.

\begin{lemma} \label{lemma_Hfour_splitting_strategy}
    The following set is suitable for any theory $T$:
    $$
    \dd_{\operatorname{triv}} := \Set{ \psi(\uxvec; \uz; \uw) \in L : \parbox{10.0cm}{``for all $\ud \in \mm \models T$, the formula $\psi(\uxvec; \uz; \ud)$ is either inconsistent\\ ${}$\hspace{2pt} or implies no finite disjunction of non-trivial linear dependencies\\ ${}$\hspace{2pt} in $\uxvec$ over $\VV$''}}.
    $$
    If $T$ furthermore satisfies \Hfour{}, then $\dd_{\operatorname{triv}}$ also has a splitting strategy.
\begin{proof}
    It is clear that $\dd_{\operatorname{triv}}$ is suitable.
    Now assume that \Hfour{} holds and that an $L$-formula $\psi(\uxvec; \uz; \uw)$ is given with $\ux = (x_1, \dots, x_n)$.
    By Lemma \ref{lemma_hfour_fml} (and accounting for the fact that we now work with $\uxvec$ instead of $\ux$), we can choose
    $$
    \sigma_{\exists \uz :\psi(\uxvec; \uz; \uw)}(\uw) = \exists \uxvec \in \VV : \exists \uz : \psi(\uxvec; \uz; \uw) \wedge \bigwedge\nolimits_{k=1}^m \neg\varphi_k\Big(\sum\nolimits_{l=1}^n \sum\nolimits_{i=0}^{\deg(\lambda_{k, l})} (\lambda_{k, l})_i \cdot x^i_l; \uw\Big),
    $$
    where each $(\lambda_{k, 1}, \dots, \lambda_{k, n}) \in K[X]^n \setminus \set{\uzero}$ is a non-zero tuple of polynomials and each $\varphi_k(y; \uw)$ is algebraic in $y$.
    For every $\ii \subseteq \set{1, \dots, m}$, define
    \begin{align*}
        \psi_\ii(\uxvec; \uz; \uw) := \psi(\uxvec; \uz; \uw) &\wedge \bigwedge\nolimits_{k\in \ii} \varphi_k\Big(\sum\nolimits_{l=1}^n \sum\nolimits_{i=0}^{\deg(\lambda_{k, l})} (\lambda_{k, l})_i \cdot x^i_l; \uw\Big)\\
        &\wedge \bigwedge\nolimits_{k \in \set{1, \dots, m} \setminus \ii} \neg \varphi_k\Big(\sum\nolimits_{l=1}^n \sum\nolimits_{i=0}^{\deg(\lambda_{k, l})} (\lambda_{k, l})_i \cdot x^i_l; \uw\Big).
    \end{align*}
    Clearly, $\psi(\uxvec; \uz; \uw) \equiv \dot\bigvee_{\ii \subseteq \set{1, \dots, m}} \psi_\ii(\uxvec; \uz; \uw)$.
    For $\ii \neq \varnothing$, we immediately see that $\psi_\ii(\uxvec; \uz; \uw)$ is as in (ii) of Definition \ref{def_suitable_set} (note that we can apply Lemma \ref{lemma_hfour_fml} such that each $(\lambda_{k, l})_i$ is zero if $x^i_l$ does not appear in $\psi(\uxvec; \uz; \uw)$).
    It now suffices to show that $\psi_\varnothing(\uxvec; \uz; \uw) \in \dd_{\operatorname{triv}}$ in order to prove that $\dd_{\operatorname{triv}}$ has a splitting strategy.
    For this, suppose that $\ud \in \mm \models T$ is given and that $\psi_\varnothing(\uxvec; \uz; \ud)$ is consistent.
    This means that $\mm \models \sigma_{\exists \uz :\psi(\uxvec; \uz; \uw)}(\ud)$ holds, and hence that $\exists \uz :\psi(\uxvec; \uz; \ud)$ implies no finite disjunction of non-trivial linear dependencies in $\uxvec$ over $\VV$.
    This is equivalent to $\psi(\uxvec; \uz; \ud)$ not implying any finite disjunction of non-trivial linear dependencies in $\uxvec$ over $\VV$.
    Since we have $\psi(\uxvec; \uz; \ud) \equiv \dot\bigvee_{\ii \subseteq \set{1, \dots, m}} \psi_\ii(\uxvec; \uz; \ud)$ and, for $\ii \neq \varnothing$, each $\psi_\ii(\uxvec; \uz; \ud)$ implies some non-trivial linear dependencies in $\uxvec$ over $\VV$, we conclude that $\psi_\varnothing(\uxvec; \uz; \ud)$ does not imply any finite disjunction of non-trivial linear dependencies in $\uxvec$ over $\VV$.
    This shows $\psi_\varnothing(\uxvec; \uz; \uw) \in \dd_{\operatorname{triv}}$ by definition.
\end{proof}
\end{lemma}

\noindent One can actually show that the existence of a suitable set with a splitting strategy is equivalent to \Hfour{}.

\subsection{Proof of Theorem \ref{theorem_big_fml_preceise}}

The proof of Theorem \ref{theorem_big_fml_preceise} is very similar to the proof of the direction ``$\Leftarrow$'' from Theorem \ref{theorem_big_characterization}, as presented in \cite{Chi25} (Theorem 3.34 there).
The key difference is that, in \cite{Chi25}, we worked with existential $L_\theta(M)$-formulas, while here we do not plug parameters into the formula.
This creates the need for algebraic patterns and \Hfour{}, or equivalently, for the existence of a suitable set $\dd$ with a splitting strategy.

\begin{definition}
    We define $\LRC$ as the language of (left) $R_C$-modules (with $R_C$ as in Fact \ref{theorem_r_c_def}), i.e., $\LRC = (0, +, (r)_{r \in R_C})$, where each $r \in R_C$ is treated as a unary function symbol.
\end{definition}

\noindent We will write $r(x)$ instead of $r \cdot x$, since $r$ will usually be a function such as $\pi_{\Image(F^C)}$.
Given a model $(\VV, \theta) \models \TKvsThe \cup \set{\text{``$\theta$ is $C$-image-complete''}}$, we define an $\LRC$-structure on $\VV$ using the definable functions of which $R_C$ consists (see the paragraph below Fact \ref{fact_when_field}).
We now define transformability, the main tool used in the proof of ``$\Leftarrow$'' of Theorem \ref{theorem_big_characterization}:

\begin{definition} \label{def_transfo}
    Let $E(\ux; \uy)$ and $E'(\ux'; \uy)$ be two conjunctions of $\LRC$-equations.
    We say that $E(\ux; \uy)$ is \textbf{transformable} into $E'(\ux'; \uy)$ if both
    \begin{enumerate}[(i)]
        \item $\TKvsThe \cup \set{\text{``$\theta$ is $C$-image-complete''}} \models \forall \ux\uy: E(\ux; \uy) \rightarrow E'(\underline{\nu}(\ux); \uy) \wedge \ux = \utau(\underline{\nu}(\ux); \uy)$,
        \item $\TKvsThe \cup \set{\text{``$\theta$ is $C$-image-complete''}} \models \forall \ux'\uy: E'(\ux'; \uy) \rightarrow E(\utau(\ux'; \uy); \uy)$
    \end{enumerate}
    hold for some tuple $\underline{\nu}(\ux)$ of $\LRC$-terms and some tuple $\utau(\ux'; \uy)$ with each entry $\tau_k(\ux'; \uy)$ being the sum of an $\LKThe$-term in $\ux'$ and an $\LRC$-term in $\uy$.
    We say that the tuples $\underline{\nu}(\ux)$ and $\utau(\ux'; \uy)$ \textbf{witness} the transformability.
\end{definition}

\noindent When $\utau(\ux'; \uy)$ does not depend on $\uy$, we may simply write $\utau(\ux')$. Transformability can be thought of as a stronger form of ``equivalence up to existence''.
This is because $E(\ux; \uy)$ being transformable into $E'(\ux'; \uy)$ implies 
$$
\exists \ux \in \VV : \phi(\ux; \uy) \wedge E(\ux; \uy) \quad \equiv\quad \exists \ux' \in \VV : \phi(\utau(\ux'; \uy); \uy) \wedge E'(\ux'; \uy)
$$ for every formula $\phi(\ux; \uy)$.% In many cases, $E'(\ux'; \uy)$ will be of the form $\varphi(\uy) \wedge S(\ux'; \mu(\uy))$, where $S(\ux', \tiluy)$ is a parametrized $C$-sequence-system and $(\varphi, \mu)$ is a pair compatible with $S$, as in the following definition: 

%\begin{definition}
%    \label{def_compatible_pair} Let $S(\ux; \tiluy) = \bigwedge_{k=1}^n f_k^{q_k}[\theta](x_{\ldd, k}) = \Tilde{y}{}_k$ be a parametrized $C$-sequence-system as in Definition \ref{def_param_c_sequence_system}.
%    Let $\varphi(\uy)$ be a conjunction of $\LRC$-equations, and let $\umu{}(\uy)$ be a tuple of $\LRC$-terms, where $\uy$ is another tuple of variables.
%    If
%    $$
%    \TKvsThe \cup \set{\text{``$\theta$ is $C$-image-complete''}} \models \forall \uy \in \VV : \varphi(\uy) \rightarrow \text{``$\umu{}(\uy)$ is compatible with $S$''}
%    $$
%    holds, we say that the pair $(\varphi, \umu{})$ is \textbf{compatible} with the parametrized $C$-sequence-system $S$.
%\end{definition}

\noindent The following three facts have already been shown in \cite{Chi25}:

\begin{fact}[Corollary 3.6 in \cite{Chi25}]
    \label{corollary_equiv_exist_form}
    Modulo $T_\theta$, any existential $L_\theta$-formula $\phi(\uw)$ is equivalent to a formula of the form $\exists \ux \in \VV : \psi_\theta(\ux; \uw)$ for some $L$-formula $\psi(\uxvec; \uw)$ (recall our \placeholderNotation{}: $\psi_\theta(\ux; \uw)$ is $\psi(\uxvec; \uw)$ with every placeholder variable $x^i_k$ replaced by $\theta^i(x_k)$). 
\end{fact}

\begin{fact}[Lemma 4.13 in \cite{Chi25}] \label{lemma_trafo_top}
    The formula $\top$ is, as a formula in $\ux$ and $\uy$, transformable into a formula of the form
    $$
    S(\ux'; \uzero),
    $$
    where $S(\ux'; \uy')$ is a parametrized $C$-sequence-system, and $\uzero$ is treated as a tuple of $\LKThe$-terms in $\uy$. Note that the tuple $\uzero$ is compatible with any parametrized $C$-sequence-system.
\end{fact}

\begin{fact}[Lemma 4.6 in \cite{Chi25}] \label{lemma_main_trafo}
    Suppose a conjunction of $\LRC$-equations $E(\ux; \uy)$ is transformable into a formula of the form
    $
    \varphi(\uy) \wedge S(\ux'; \umu{}(\uy)),
    $
    where $S(\ux'; \tiluy)$ is a parametrized $C$-sequence-system, $\umu{}(\uy)$ is a tuple of $\LRC$-terms, and $\varphi(\uy)$ is a conjunction of $\LRC$-equations.
    Given any $L$-formula $\psi(\uxvec; \uw)$, there is another $L$-formula $\psi'(\uxvec{}'; \uw\tiluw)$ that is bounded by $S$, and a tuple of $\LRC$-terms $\ut(\uy)$ such that
    $$
    \exists \ux \in \VV : \psi_\theta(\ux; \uw) \wedge E(\ux; \uy) \quad \equiv \quad \varphi(\uy) \wedge \exists \ux' \in \VV : \psi'_\theta(\ux'; \uw\ut(\uy)) \wedge S(\ux'; \umu(\uy))
    $$
    holds in $T_\theta \cup \set{\text{``$\theta$ is $C$-image-complete''}}$.
\end{fact}

\noindent The following lemma tells us that the conjunction of a parametrized $C$-sequence-system and some $\LKThe$-equations is again essentially transformable into a parametrized $C$-sequence-system whose rank-degree pair (recall Definition \ref{def_rk_deg_def}) is lexicographically less then or equal to that of the orignal $C$-sequence-system.
As in \cite{Chi25}, it will help us to prove Theorem \ref{theorem_big_fml_preceise} via induction on $(\rk(S), \deg(S))$.

\begin{fact}[Transformation Lemma, Lemma 4.11 in \cite{Chi25}]\label{lemma_trafoo}
    Let $S(\ux; \uy{}_1)$ be a parametrized $C$-sequence-system, let $E(\ux; \uy{}_2)$ be a conjunction of $\LKThe$-equations, and let $\umu{}_1(\uy)$ and $\umu{}_2(\uy)$ be tuples of $\LRC$-terms.
    Then there is another parametrized $C$-sequence-system $S'(\ux'; \uy')$, a conjunction $\varphi'(\uy)$ of $\LRC$-equations, and a tuple $\umu{}'(\uy)$ of $\LRC$-terms such that $\varphi'(\uy)$ implies that $\umu'(\uy)$ is compatible with $S'$, and
    $$
    \text{$S(\ux; \umu{}_1(\uy)) \wedge E(\ux; \umu{}_2(\uy))\quad$ is transformable into $\quad\varphi'(\uy) \wedge S'(\ux'; \umu{}'(\uy))$}.
    $$
    Furthermore, $(\rk(S'), \deg(S')) \leq_\Lex (\rk(S), \deg(S))$ holds, and if $E(\ux; \uy{}_2)$ contains at least one equation that is non-trivial in $\ux$ and bounded by $S$, then this inequality is strict.
\end{fact}

\noindent The next ingredient is a lemma that, roughly speaking, allows us to insert $\LRC$-equations and terms into algebraic patterns:

\begin{lemma} \label{lemma_extend_alg_pattern_by_eq_and_tuples}
    Let $Y$ be an algebraic pattern in $\uw$ defined by $\psi_Y(\uy; \uw)$, let $\varphi(\uy)$ be a conjunction of $\LRC$-equations, and let $\ut(\uy)$ be a tuple of $\LRC$-terms.
    Then there is another algebraic pattern $Y'$ in $\uw$ defined by $\psi_{Y'}(\uy'; \uw)$ with $\uy' = \uy\uy''$ for another tuple of variables $\uy''$, and a coordinate projection $\pi$ such that:
    \begin{enumerate}[(i)]
        \item $T_\theta \cup \set{\text{``$\theta$ is $C$-image-complete''}} \models \forall \uy\uw : \big(\uy \in Y_\uw \wedge \varphi(\uy)\big) \leftrightarrow \exists \uy'' : \uy\uy'' \in Y'_\uw$; and
        \item $T_\theta \cup \set{\text{``$\theta$ is $C$-image-complete''}} \models \forall \uy'\uw : \uy' \in Y'_\uw \rightarrow \pi(\uy') = \ut(\uy)$.
    \end{enumerate}
\begin{proof}
    Write $\uy = (y_1, \dots, y_n)$, $\varphi(\uy) = \bigwedge_{k=1}^{m_\varphi} \sum_{l=1}^n r_{\varphi, k, l}(y_l) = 0$, and $\ut(\uy) = (t_1(\uy), \dots, t_{m_t}(\uy))$ with $t_k(\uy) := \sum_{l=1}^n r_{t, k, l}(y_l)$.
    Define the tuples
    $$
    \uy_r := (y_{r, *, k, l} : * \in \set{\varphi, t}, 1 \leq k \leq m_*, 1 \leq l \leq n),
    $$
    $\uy_\varphi := (y_{\varphi, 1}, \dots, y_{\varphi, m_\varphi})$, $\uy_t := (y_{t, 1}, \dots, y_{t, m_t})$, and $\uy' := \uy\uy_r\uy_\varphi\uy_t$.
    The formula
    \begin{align*}
        \psi_{Y'}(\uy'; \uw) := \psi_Y(\uy; \uw) &\wedge \bigwedge\nolimits_{* \in \set{\varphi, t}}\bigwedge\nolimits_{k=1}^{m_*}\bigwedge\nolimits_{l=1}^n \exists x \in \VV : y_{r, *, k, l} = r_{*, k, l}(x) \wedge x = y_l \\
        & \wedge \bigwedge\nolimits_{k=1}^{m_\varphi} \exists x \in \VV : y_{\varphi, k} = 1[\theta](x) \wedge \Big(x = 0 \wedge x = \sum\nolimits_{l=1}^n y_{r, \varphi, k, l}\Big) \\
        & \wedge \bigwedge\nolimits_{k=1}^{m_t} \exists x \in \VV : y_{t, k} = 1[\theta](x) \wedge x = \sum\nolimits_{l=1}^n y_{r, t, k, l}
    \end{align*}
    defines an algebraic pattern.
    Together with the coordinate projection $\pi$ given by $\pi(\uy') = \uy_t$, it satisfies (i) and (ii).
\end{proof}
\end{lemma}

\begin{proof}[Proof of Theorem \ref{theorem_big_fml_preceise}]
    \newcounter{savedtheorem}
    \setcounter{savedtheorem}{\value{theorem}}
    \setcounter{theorem}{\value{savedtheorem55}}
    We start with two claims that are very similar to Claim 4.14.1 and Claim 4.14.2 in \cite{Chi25}:
\begin{subclaim} \label{lemma_exist_fml_reduc}
    Any existential $L_\theta$-formula $\phi(\uz; \uw)$ is, modulo $T_\theta \cup \set{\text{``$\theta$ is $C$-image-complete''}}$, equivalent to a formula of the form
    \begin{align*}
        \exists \uy \in Y_\uw : \exists \ux \in \VV : \psi_\theta(\ux; \uz; \uy\uw) \wedge S(\ux; \pi(\uy))
    \end{align*}
    with $S(\ux; \tiluy)$ being a parametrized $C$-sequence-system, $\psi(\uxvec; \uz; \uy\uw)$ being an $L$-formula bounded by $S$, and $(Y, \pi)$ being compatible with $S$.
\begin{innerproof}
    By Fact \ref{corollary_equiv_exist_form}, the formula $\phi(\uz; \uw)$ is, modulo $T_\theta$, equivalent to a formula of the form $\exists \ux \in \VV : \psi_\theta(\ux; \uz; \uw)$.
    By Fact \ref{lemma_trafo_top}, $\top$, as a formula in $\ux$, is transformable into $S(\ux'; \uzero)$, where $S(\ux'; \tiluy)$ is a parametrized $C$-sequence-system.
    Hence, by Fact \ref{lemma_main_trafo}, we see that $\phi(\uz; \uw)$ is equivalent to a formula of the form
    $$
    \exists \ux' \in \VV : \psi'_\theta(\ux'; \uz; \uzero\uw) \wedge S(\ux'; \uzero)
    $$
    where $\psi'(\uxvec{}'; \uz; \uy\uw)$ is an $L$-formula bounded by $S$.
    As $0$ is $L$-definable, we can assume, without loss of generality, that the two tuples $\uzero$ appearing above have the same length $m \in \NN$.
    Now define an algebraic pattern $Y$ using $\psi_Y(y_1, \dots, y_m; \uw) := \bigwedge_{k=1}^m \exists x \in \VV : y_k = 1[\theta](x) \wedge x = 0$.
    It is easy to see that $\phi(\uz;\uw)$ is equivalent to
    $$
    \exists \uy \in Y_\uw : \exists \ux' \in \VV : \psi'_\theta(\ux'; \uz; \uy\uw) \wedge S(\ux'; \Id(\uy))
    $$
    modulo $T_\theta \cup \set{\text{``$\theta$ is $C$-image-complete''}}$.
    Clearly, $(Y, \Id)$ is compatible with $S$, as $\uzero$ is compatible with any parametrized $C$-sequence-system.
\end{innerproof}
\end{subclaim}

\begin{subclaim} \label{claim_subsub}
    Let $\phi(\uz; \uw) := \exists \uy \in Y_\uw : \exists \ux \in \VV : \psi_\theta(\ux; \uz; \uy\uw) \wedge S(\ux; \pi(\uy)) \wedge E(\ux; \uy)$ be an $L_\theta$-formula, where $S(\ux; \tiluy)$ is a parametrized $C$-sequence-system, $\psi(\uxvec; \uz; \uy\uw)$ is an $L$-formula, $E(\ux; \uy)$ is a non-empty conjunction of non-trivial in $\ux$ $\LKThe$-equations that are bounded by $S$, $Y$ is an algebraic pattern, and $\pi$ is a coordinate projection such that $(Y, \pi)$ is compatible with $S$.
    Then $\phi(\uz; \uw)$ is, modulo $T_\theta \cup \set{\text{``$\theta$ is $C$-image-complete''}}$, equivalent to a formula of the form
    $$
    \exists \uy' \in Y'_\uw : \exists \ux' \in \VV : \psi'_\theta(\ux'; \uz; \uy'\uw) \wedge S'(\ux'; \pi'(\uy'))
    $$
    where $S'(\ux'; \tiluy{}')$ is another parametrized $C$-sequence-system with $(\rk(S'), \deg(S')) \!<_\Lex\! (\rk(S), \deg(S))$, $\psi'(\uxvec{}'; \uz; \uy'\uw)$ is an $L$-formula bounded by $S'$, $Y'$ is an algebraic pattern, and $\pi'$ is a coordinate projection such that $(Y', \pi')$ is compatible with $S'$.
\begin{innerproof}
    Since $E(\ux; \uy)$ contains a non-trivial $\LKThe$-equation bounded by $S$, our \TrafoLemma{} yields that $S(\ux; \pi(\uy)) \wedge E(\ux; \uy)$ is transformable into a formula of the form
    $$
    \varphi(\uy) \wedge S'(\ux';\umu{}'(\uy))
    $$
    where $S'(\ux'; \tiluy{}')$ is a parametrized $C$-sequence-system satisfying $(\rk(S'), \deg(S')) <_\Lex (\rk(S), \deg(S))$. Furthermore, the conjunction of $\LRC$-equations $\varphi(\uy)$ implies that the tuple of $\LRC$-terms $\umu{}'(\uy)$ is compatible with $S'$.
    By Fact \ref{lemma_main_trafo} and this transformability, $\phi(\uz; \uw)$ is equivalent, modulo the theory $T_\theta \cup \set{\text{``$\theta$ is $C$-image-complete''}}$, to a formula of the form
    \begin{align}
        \exists \uy \in Y_\uw : \exists \ux' \in \VV : \varphi(\uy) \wedge \psi^*_{\theta}(\ux'; \uz; \ut(\uy)\uw) \wedge S'(\ux'; \umu{}'(\uy)) \label{tag_fml_ugly}
    \end{align}
    where $\psi^*(\uxvec{}'; \uz; \uy^*\uw)$ is an $L$-formula bounded by $S'$, and $\ut(\uy)$ is a tuple of $\LRC$-terms.
    To apply Fact \ref{lemma_main_trafo}, note that, since $\uy \in Y_\uw$ implies $\uy \in \VV$, we may assume that $\uy = 1[\theta](\uy)$ is already a subtuple of $\ut(\uy)$, so $\uy$ does not need to appear separately in $\psi^*(\uxvec{}'; \uz; \uy^*\uw)$.
    Now apply Lemma \ref{lemma_extend_alg_pattern_by_eq_and_tuples} to $Y$, $\varphi(\uy)$, and $\ut(\uy)\umu{}'(\uy)$ to obtain an algebraic pattern $Y'$ in $\uw$, defined by a formula $\psi_{Y'}(\uy'; \uw)$ with $\uy' = \uy \uy''$, such that:
    \begin{enumerate}[(i)]
        \item $T_\theta \cup \set{\text{``$\theta$ is $C$-image-complete''}} \models \forall \uy\uw : \big(\uy \in Y_{\uw} \wedge \varphi(\uy)\big) \leftrightarrow \exists \uy'' : \uy\uy'' \in Y'_\uw$; and
        \item $T_\theta \cup \set{\text{``$\theta$ is $C$-image-complete''}} \models \forall \uy'\uw : \uy' \in Y'_\uw \rightarrow \big(\pi'_0(\uy') = \ut(\uy) \wedge \pi'(\uy') = \umu{}'(\uy)\big)$, where both $\pi'_0$ and $\pi'$ are coordinate projections.
    \end{enumerate}
    It follows that the formula in (\ref{tag_fml_ugly}) is, modulo $T_\theta \cup \set{\text{``$\theta$ is $C$-image-complete''}}$, equivalent to
    $$
    \exists \uy' \in Y'_\uw : \exists \ux' \in \VV : \psi'_{\theta}(\ux'; \uz; \uy'\uw) \wedge S'(\ux'; \pi'(\uy'))
    $$
    where $\psi'(\uxvec{}'; \uz; \uy' \uw) := \psi^*(\uxvec{}'; \uz; \pi'_0(\uy')\uw)$ is still an $L$-formula bounded by $S'$.
\end{innerproof}
\end{subclaim}
\noindent We can now begin the actual proof of Theorem \ref{theorem_big_fml_preceise}.
By Claim \ref{lemma_exist_fml_reduc}, it suffices to show the theorem for formulas $\phi(\uz; \uw)$ of the form
    $$
    \exists \uy \in Y_\uw : \exists \ux \in \VV : \psi_\theta(\ux; \uz; \uy\uw) \wedge S(\ux; \pi(\uy))
    $$
    with $S(\ux; \tiluy)$ being a parametrized $C$-sequence-system, $\psi(\uxvec; \uz; \uy\uw)$ being an $L$-formula bounded by $S$, and $(Y, \pi)$ being compatible with $S$.
    Since we do not need the partition of $\ux$ into $\ux_\li$ and $\ux_\ld$, we simply write $\ux = (x_1, \dots, x_n)$.
    We prove the theorem by induction on $(\rk(S), \deg(S))$.

    Fix such a formula $\phi(\uz;\uw)$ and assume that we have already shown the theorem for all formulas of this form
    $$
    \exists \uy_0 \in Y_{0,\uw} : \exists \ux_0 \in \VV : \psi_{0,\theta}(\ux_0; \uz; \uy_0\uw) \wedge S_0(\ux_0; \pi_0(\uy_0))
    $$
    with $(\rk(S_0), \deg(S_0)) <_\Lex (\rk(S), \deg(S))$.
    As $\dd$ is, by assumption in Theorem \ref{theorem_big_fml_preceise}, a suitable set with a splitting strategy, we can apply this splitting strategy to $\psi(\uxvec; \uz; \uy\uw)$.
    This yields
    $$
    T \models \forall\uz\uy\uw : \forall\uxvec\in \VV : \psi(\uxvec; \uz;\uy\uw) \leftrightarrow \bigvee\nolimits_{k=1}^m \psi_k(\uxvec; \uz; \uy\uw)
    $$
    where, for any $k \in \set{1, \dots, m}$, the variables that appear in $\psi_k(\uxvec; \uz;\uy\uw)$ are a subset of the variables that appear in $\psi(\uxvec; \uz;\uy\uw)$, and either:
    \begin{enumerate}[(i)]
        \item $\psi_k(\uxvec; \uz; \uy\uw)$ lies in $\dd$, or
        \item there is an $L$-formula $\varphi(y_*; \uy\uw)$ that defines finite subsets of $\VV$ in $y_*$, and a tuple of polynomials $\ulambda = (\lambda_{1}, \dots, \lambda_{n}) \in K[X]^{n} \setminus \set{ \uzero}$ (both $\varphi(y_*;\uy \uw)$ and $\ulambda$ may depend on $k$), such that the following holds:
        $$
            T \models \forall\uz\uy\uw : \forall \uxvec \in \VV :  \psi_k(\uxvec; \uz; \uy\uw) \rightarrow \varphi\Big(\! \sum\nolimits_{l=1}^{n}\sum\nolimits_{i=0}^{\deg(\lambda_{l})} (\lambda_{l})_i \cdot x^i_l; \uy\uw \Big).
        $$
    \end{enumerate}
    Now $\phi(\uz; \uw)$ is clearly, modulo $T_\theta \cup \set{\text{``$\theta$ is $C$-image-complete''}}$, equivalent to the disjunction
    $$
    \bigvee\nolimits_{k=1}^m \exists \uy \in Y_\uw : \exists \ux \in \VV : \psi_{k,\theta}(\ux; \uz; \uy\uw) \wedge S(\ux; \pi(\uy)).
    $$
    After discarding inconsistent disjuncts, we may assume that every $\psi_{k,\theta}(\ux; \uz; \uy\uw)$ is consistent with $\uy \in \VV$.
    It suffices to show that Theorem \ref{theorem_big_fml_preceise} holds for every disjunct in the formula above.
    For this, fix some $k \in \set{1, \dots, m}$ and distinguish between cases (i) and (ii).

    If case (i) holds for $k$, then $\exists \uy \in Y_\uw : \exists \ux \in \VV : \psi_{k,\theta}(\ux; \uz; \uy\uw) \wedge S(\ux; \pi(\uy))$ is exactly a formula as described in the statement of Theorem \ref{theorem_big_fml_preceise}.
    The boundedness of $\psi_k(\uxvec; \uz; \uy\uw)$ follows because $\psi(\uxvec; \uz; \uy\uw)$ is bounded by $S$ and $\psi_k(\uxvec; \uz; \uy\uw)$ only uses variables that appear in $\psi(\uxvec; \uz; \uy\uw)$.

    Now assume that case (ii) holds for $k$, and let $\varphi(y_*; \uy\uw)$ and $\ulambda{}$ be given as described in (ii).
    Also, let $\psi_Y(\uy;\uw)$ be the formula defining the algebraic pattern $Y$.
    Set $\uy' := \uy y_*$.
    The formula
    $$
    \psi_{Y'}(\uy'; \uw) := \psi_Y(\uy; \uw) \wedge \exists x \in \VV : \varphi(x; \uy\uw) \wedge y_* = 1[\theta](x)
    $$
    defines another algebraic pattern $Y'$ in $\uw$.
    Using the implication in (ii) and our \placeholderNotation{}, we see that $\exists \uy \in Y_\uw : \exists \ux \in \VV : \psi_{k,\theta}(\ux; \uz; \uy\uw) \wedge S(\ux; \pi(\uy))$ is equivalent to
    \begin{align}
        \exists \uy' \in Y'_{\uw} : \exists \ux \in \VV : \psi_{k,\theta}(\ux; \uz; \uy\uw) \wedge S(\ux; \pi(\uy)) \wedge \sum\nolimits_{l=1}^{n} \lambda_{l}[\theta](x_l) = y_* \label{tag_form_even_ugler}
    \end{align}
    Let $E(\ux; y_*)$ denote the equation on the right.
    A simple but crucial observation is that $E(\ux; y_*)$ is bounded by $S$.
    First, $\psi_k(\uxvec; \uz; \uy \uw)$ is bounded by $S$, since $\psi(\uxvec; \uz; \uy\uw)$ is bounded by $S$ and the variables that appear in $\psi_k(\uxvec; \uz; \uy \uw)$ are a subset of those that appear in $\psi(\uxvec; \uz; \uy\uw)$.
    Also, the implication in (ii) above cannot hold if the coefficient $(\lambda_l)_i$ is non-zero for some $(i, l)$ for which $x^i_l$ does not appear in $\psi_k(\uxvec; \uz; \uy\uw)$, because this formula is consistent.
    Since the tuple $\ulambda$ is non-zero, the $\LKThe$-equation $E(\ux; y_*)$ is also non-trivial.
    We can now apply Claim \ref{claim_subsub} to the formula (\ref{tag_form_even_ugler}), and then apply the induction hypothesis to conclude that Theorem \ref{theorem_big_fml_preceise} holds for the $k$-th disjunct.
    This finishes the induction step.

    It remains to note that in the base case, $(\rk(S), \deg(S)) = (0, 0)$, the tuple $\ux$ is empty.
    When we apply the splitting strategy to the formula $\psi(\uxvec{}; \uz; \uy\uw)$ in this case, (i) must hold for every $\psi_k(\uxvec{}; \uz; \uy\uw)$, since $K[X]^{0} \setminus \set{\uzero}$ is empty.
    \setcounter{theorem}{\value{savedtheorem}}
\end{proof}

\section{Applications of our Characterization of Definable Sets}
\label{sec_completions_etc}

In this section, we apply our characterization of definable sets from Theorem \ref{theorem_big_fml} to answer several basic questions about models of $T\theta^C$:
When do two tuples have the same type?
What are the completions of $T\theta^C$?
When is $T\theta^C$ complete?
What does the algebraic closure of a set look like?
We start with two basic definitions:

\begin{definition} \label{def_span_c}
    Let $(\VV, \theta)$ be $C$-image-complete.
    For every $U \subseteq \VV$, we define
    $$
    \spanA{U}{C} := \Set{\sum\nolimits_{i=1}^q r_i(u_i) : q \in \NN, r_1, \dots, r_q \in R_C, u_1, \dots, u_q \in U},
    $$
    where $R_C$ is the ring of definable scalars from Fact \ref{theorem_r_c_def}.
    Equivalently, $\spanA{U}{C}$ is the smallest subset of $\VV$ that contains $U$ and is closed under sums and the functions in $R_C$.
\end{definition}

\begin{definition} \label{def_cl_theater}
    Given $A \subseteq M$, where $(\mm, \theta) \models T_\theta \cup \set{\text{``$\theta$ is $C$-image-complete''}}$, we define $\cl_\theta(A)$ as the smallest set that contains $A$ and is closed under both $\acl_L$ and $\spanA{\;}{C}$.
\end{definition}

\noindent It is straightforward to check that $\spanA{\;}{C}$ is a closure operator, and hence so is $\cl_\theta$.
Moreover, $\cl_\theta(A)$ is closed under $\spanA{\;}{L}$, where $\spanA{B}{L}$ denotes the $L$-substructure generated by $B$.
Thus $\cl_\theta(A)$ is an $L$-substructure of $\mm$.
It is also closed under $\theta$:

\begin{notation}
    We write $(\cl_\theta(A), \theta)$ for the $L_\theta$-substructure $(\cl_\theta(A), \theta_{\restriction \cl_\theta(A)}) \subseteq (\mm, \theta)$.
\end{notation}

\noindent In this chapter, we consider maps from one model of $T\theta^C$ into another that are both $L$-elementary, i.e., preserve $L$-types, and $L_\theta$-isomorphisms.
We give two criteria under which $L$-elementarity follows from being an $L_\theta$-isomorphism.

\begin{remark} \label{remark_when_L_theta_iso_is_enough}
    Let $(\mm_1, \theta_1), (\mm_2, \theta_2) \models T\theta^C$, let $A_i \subseteq M_i$ for $i \in \set{1, 2}$, and let
    $$
    \iota \colon (\cl_\theta^{(\mm_1, \theta_1)}(A_1), \theta_1) \to (\cl_\theta^{(\mm_2, \theta_2)}(A_2), \theta_2)
    $$
    be an $L_\theta$-isomorphism.
    Then $\iota$ is also an $L$-elementary map with respect to $\mm_1$ and $\mm_2$ if one of the following conditions holds:
    \begin{enumerate}[(i)]
        \item $T$ admits quantifier elimination.
        \item For every model $\mm \models T$ and every subset $A \subseteq M$, we have $\acl_L(A) \models T$.
    \end{enumerate}
\begin{proof}
    If $T$ admits quantifier elimination, then any $L$-isomorphism between $L$-substructures is $L$-elementary.
    If condition (ii) holds, then $\cl_\theta^{(\mm_i, \theta_i)}(A_i)$ is a model of $T$ for each $i \in \set{1, 2}$.
    Hence the $L$-isomorphism $\iota$ is $L$-elementary with respect to the models $\cl_\theta^{(\mm_1, \theta_1)}(A_1)$ and $\cl_\theta^{(\mm_2, \theta_2)}(A_2)$.
    Since $T$ is model-complete, the inclusions of these models into $\mm_1$ and $\mm_2$, respectively, are elementary.
    Thus $\iota$ is $L$-elementary with respect to $\mm_1$ and $\mm_2$.
\end{proof}
\end{remark}

\noindent As above, an $L$-elementary map always means an $L$-elementary map with respect to the ambient $L$-models $\mm_1, \mm_2$.
It does not mean an $L$-elementary map with respect to, for example, the $L$-structures $\cl_\theta^{(\mm_1, \theta_1)}(A_1)$ and $\cl_\theta^{(\mm_2, \theta_2)}(A_2)$.

\subsection{Types in $T\theta^C$}

\begin{lemma} \label{lemma_iota_alg_pattern}
    Let $A$ and $B$ be $\cl_\theta$-closed subsets of $(\mm_1, \theta_1) \models T\theta^C$ and $(\mm_2, \theta_2) \models T\theta^C$, respectively.
    Let $\iota \colon (A, \theta_1) \to (B, \theta_2)$ be an $L_\theta$-isomorphism that is an $L$-elementary map with respect to $\mm_1$ and $\mm_2$.
    Then
    $$
    \uu \in Y_\ua \quad \Leftrightarrow \quad \iota(\uu) \in Y_{\iota(\ua)}
    $$
    holds for every algebraic pattern $Y$ and all tuples $\ua \in A$ and $\uu \in A$.
    Moreover, if $\uu \not\in A$, then $\uu \not\in Y_\ua$.
\begin{proof}
    By Definition \ref{def_alg_pattern}, the algebraic pattern $Y$ is given by an $L_\theta$-formula of the form
    $$
    \psi_Y(y_1, \dots, y_m; \uw) := \bigwedge\nolimits_{k=1}^m \exists x \!\in\! \VV : y_k = r_k(x) \wedge \varphi_k(x; y_1, \dots, y_{k-1}; \uw),
    $$
    where each $r_k$ is an element of the ring $R_C$, and each formula $\varphi_k(x; y_1, \dots, y_{k-1}; \uw)$ is an $L$-formula that is algebraic in $x$.
    Assume $\uu \in Y_\ua$, and let $v_1, \dots, v_m \in \VV_1$ be such that
    $$
    (\mm_1, \theta_1) \models \bigwedge\nolimits_{k=1}^m u_k = r_k(v_k) \wedge \varphi_k(v_k; u_1, \dots, u_{k-1}; \ua).
    $$
    We show by induction on $k$ that $u_k \in A$ and $v_k \in A$ for all $k \in \set{1, \dots, m}$.
    Suppose $u_1, \dots, u_{k-1} \in A$.
    Since $\varphi_k(x; u_1, \dots, u_{k-1}; \ua)$ is an $L$-formula that is algebraic in $x$, we have $v_k \in \acl_L(u_1, \dots, u_{k-1}, \ua) \subseteq A$.
    Therefore, $u_k = r_k(v_k)$ lies in $\spanA{v_k}{C} \subseteq A$.
    In particular, if $\uu \in Y_\ua$, then $\uu \in A$, which proves the moreover statement.
    Since $\iota$ is $L$-elementary, we have
    $$
    (\mm_2, \theta_2) \models \bigwedge\nolimits_{k=1}^m \varphi_k(\iota(v_k); \iota(u_1), \dots, \iota(u_{k-1}); \iota(\ua)).
    $$
    Furthermore, $\iota(u_k) = r_k(\iota(v_k))$ for every $k$, because $\iota$ is an $L_\theta$-isomorphism.
    Thus
    $$
    (\mm_2, \theta_2) \models \bigwedge\nolimits_{k=1}^m \iota(u_k) = r_k(\iota(v_k)) \wedge \varphi_k(\iota(v_k); \iota(u_1), \dots, \iota(u_{k-1}); \iota(\ua)),
    $$
    and hence $\iota(\uu) \in Y_{\iota(\ua)}$.
    The converse follows by symmetry.
\end{proof}
\end{lemma}

\begin{remark} \label{rem_elmap_closed_lol}
    If $\iota \colon \cl_\theta(\ua) \to B$ is both an $L$-elementary map and an $L_\theta$-isomorphism, then we have $B = \cl_\theta(\iota(\ua))$.
\end{remark}

\noindent We are now ready to use our characterization of definable sets to determine when two tuples have the same type. In the following, a ($\kappa$-)monster model of a theory $T$ refers to a special model $\MM \models T$ (with $\operatorname{cf}(|M|) \geq \kappa$). The definition of a special model can be found in Definition 6.1.1 in \cite{TZ12}.

\begin{corollary} \label{corollary_same_type}
    Assume that $T$ satisfies \Hfour{}. Let $\ua$ and $\ub$ be tuples in $(\mm_1, \theta_1) \models T\theta^C$ and $(\mm_2, \theta_2) \models T\theta^C$, respectively.
    Then
    $$
    \tp_{L_\theta}(\ua) \quad=\quad \tp_{L_\theta}(\ub)
    $$
    holds if and only if there is an $L_\theta$-isomorphism $\iota \colon (\cl_\theta(\ua), \theta_1) \to (\cl_\theta(\ub), \theta_2)$ that maps $\ua$ to $\ub$ and is an $L$-elementary map with respect to $\mm_1$ and $\mm_2$.
\begin{proof}
    If $\ua$ and $\ub$ have the same type, then we can extend $(\mm_1, \theta_1)$ and $(\mm_2, \theta_2)$ to two equally sized monster models $\MM_1$ and $\MM_2$ of $T\theta^C(\ua)$, where $\ua^{(\mm_2, \theta_2)} := \ub$.
    Since $\MM_1$ and $\MM_2$ are $L_\theta(\ua)$-isomorphic (see 6.1.4 in \cite{TZ12}), we obtain $\iota$ by restriction.

    Now assume that such an $\iota \colon (\cl_\theta(\ua), \theta_1) \to (\cl_\theta(\ub), \theta_2)$ is given.
    Then $|\ua| = |\ub|$.
    Let $\uw$ be a tuple of variables with $|\uw| = |\ua|$.
    By Theorem \ref{theorem_big_fml}, every $L_\theta$-formula in $\uw$ is, modulo $T\theta^C$, equivalent to a finite disjunction of formulas of the form $\exists \uy \in Y_\uw : \psi(\uy;\uw)$, where $Y$ is an algebraic pattern in $\uw$.
    Hence it is enough to show that $\ua$ and $\ub$ satisfy the same formulas of this form.
    Using Lemma \ref{lemma_iota_alg_pattern}, the equality $\iota(\ua) = \ub$, and the assumption that $\iota$ is $L$-elementary, we see that for any $\uu \in Y_\ua$ with $(\mm_1, \theta_1) \models \psi(\uu; \ua)$, we also have $\iota(\uu) \in Y_{\ub}$ and $(\mm_2, \theta_2) \models \psi(\iota(\uu); \ub)$.
    By symmetry, we conclude that
    $$
    (\mm_1, \theta_1) \models \exists \uy \in Y_\ua : \psi(\uy;\ua) \quad\Leftrightarrow\quad (\mm_2, \theta_2) \models \exists \uy \in Y_\ub : \psi(\uy;\ub)
    $$
    holds for every $L$-formula $\psi(\uy; \uw)$ and every algebraic pattern $Y$ in $\uw$.
    This completes the proof.
\end{proof}
\end{corollary}

\subsection{Completions of $T\theta^C$}

Since two models are elementarily equivalent if and only if the empty tuples have the same type, Corollary \ref{corollary_same_type} gives the following:

\begin{theorem} \label{theorem_ee_iff}
    Assume that $T$ satisfies \Hfour{}. Two models $(\mm_1, \theta_1), (\mm_2, \theta_2) \models T\theta^C$ are elementarily equivalent if and only if there is an $L_\theta$-isomorphism
    $$
    \iota \colon (\cl_\theta^{(\mm_1, \theta_1)}(\varnothing), \theta_1) \to (\cl_\theta^{(\mm_2, \theta_2)}(\varnothing), \theta_2)
    $$
    that is also $L$-elementary with respect to $\mm_1$ and $\mm_2$.
\end{theorem}

\begin{corollary} \label{corollary_when_is_compl}
    Assume that $T$ satisfies \Hfour{}. If $T$ is complete and $\acl_L(\varnothing) \cap \VV = \set{0}$, then $T\theta^C$ is complete.
\begin{proof}
    If $T$ is complete, then for any models $(\mm_1, \theta_1), (\mm_2, \theta_2) \models T\theta^C$, there is an elementary map $\iota \colon \acl_L^{\mm_1}(\varnothing) \to \acl_L^{\mm_2}(\varnothing)$.
    Since $\spanA{0}{C} = \set{0}$, we have $\acl_L^{\mm_i}(\varnothing) = \cl_\theta^{(\mm_i, \theta_i)}(\varnothing)$ for $i = 1, 2$.
    Moreover, $\iota$ maps $0$ to $0$, so it is also an $L_\theta$-isomorphism.
    The conclusion follows from Theorem \ref{theorem_ee_iff}.
\end{proof}
\end{corollary}

\noindent The converse of Corollary \ref{corollary_when_is_compl} is not true.
For example, take $T$ to be the theory $\TKvs$ of $K$-vector spaces expanded by a non-zero constant $c$.
As always, let $C_0$ be the transcendental kernel configuration with $C_0(f) = 0$ for all $f \in \Kp{}$.
Then $T\theta^{C_0}$ turns out to be $T_{K(X)\operatorname{-vs}}(c)$, with $c$ interpreted as a non-zero constant; see Observation \ref{observation_str_min}.
This theory is complete.
On the other hand, if one chooses another kernel configuration $C$ for which $R_C$ is not a field, then $T\theta^C$ will not be complete:
Take some $f \in \Kp{}$ with $C(f) > 0$, then both $T\theta^C \cup \set{c \in \Ker(f)}$ and $T\theta^C \cup \set{c \not\in \Ker(f)}$ are consistent.

\begin{corollary} \label{corollary_model_completion}
    Assume that $T$ satisfies \Hfour{}. The theory $T\theta^C$ is the model completion of the theory $T_\theta \cup \set{\text{``$\theta$ is $C$-image-complete''}}$, i.e., given $(\mm_0, \theta_0) \models T_\theta \cup \set{\text{``$\theta$ is $C$-image-complete''}}$, the theory
    $
    T\theta^C \cup \Diag(\mm_0, \theta_0)
    $
    is complete.
\begin{proof}
    Let $(\mm_1, \theta_1), (\mm_2, \theta_2) \models T\theta^C \cup \Diag(\mm_0, \theta_0)$ be given.
    We need to show that $M_0^{(\mm_1, \theta_1)}$ and $M_0^{(\mm_2, \theta_2)}$ have the same $L_\theta$-type.
    Since $\mm_0$ is closed under $\acl_L$ and $\spanA{\;}{C}$, we have
    $$
    \cl_\theta(M_0^{(\mm_i, \theta_i)}) = M_0^{(\mm_i, \theta_i)}
    $$
    for $i = 1, 2$.
    Now the map $\iota \colon M_0^{(\mm_1, \theta_1)} \to M_0^{(\mm_2, \theta_2)}$ given by $c^{(\mm_1, \theta_1)} \mapsto c^{(\mm_2, \theta_2)}$ for every $c \in M_0$ is an $L_\theta$-isomorphism.
    As in the proof of Remark \ref{remark_when_L_theta_iso_is_enough}, it is also $L$-elementary by the model completeness of $T$.
    The conclusion follows from Corollary \ref{corollary_same_type}.
\end{proof}
\end{corollary}

\noindent We end this section by giving a criterion for when $T\theta^C$ has quantifier elimination.

\begin{theorem} \label{theorem_qe}
    Assume that $T$ satisfies \Hfour{}. Let $L'$ be a language interdefinable with $L$, modulo $T$, such that $T$ has quantifier elimination in $L'$ and $\spanA{A}{L'} = \acl_L(A)$ for every $A \subseteq \mm \models T$.
    Then $T\theta^C$ has quantifier elimination in the language $L' \cup \LRC$.
\begin{proof}
    Let $\ua = (a_1, \dots, a_n)$ and $\ub = (b_1, \dots, b_n)$ be tuples from $(\mm_1, \theta_1), (\mm_2, \theta_2) \models T\theta^C$, respectively, with the same quantifier-free $L' \cup \LRC$-type.
    Consider the map $\iota$ given by
    $$
    t(\ua) \mapsto t(\ub)
    $$
    for every $L' \cup \LRC$-term $t(x_1, \dots, x_n)$.
    We claim that $\iota$ is an $L_\theta$-isomorphism between $(\cl_\theta(\ua), \theta_1)$ and $(\cl_\theta(\ub), \theta_2)$ that is also $L$-elementary.
    To see this, first note that
    $$
    \cl_\theta(\ua) = \set{t(\ua) : \text{$t(x_1, \dots, x_n)$ is an $L' \cup \LRC$-term}},
    $$
    since this is the smallest set that contains $\ua$, is closed under $\spanA{\;}{C} = \spanA{\;}{L_{R_C}}$, and is closed under $\acl_L = \spanA{\;}{L'}$.
    The analogous equality holds for $\ub$.
    Since $\ua$ and $\ub$ have the same quantifier-free $L' \cup \LRC$-type, $\iota$ is well defined and bijective.
    Since $L$ is interdefinable with $L'$ and $T$ has quantifier elimination in $L'$, the map $\iota$ is $L$-elementary.
    Moreover, $\theta$ commutes with $\iota$, so $\iota$ is an $L_\theta$-isomorphism.
    Hence $\tp_{L_\theta}(\ua) = \tp_{L_\theta}(\ub)$ by Corollary \ref{corollary_same_type}.
\end{proof}
\end{theorem}

\noindent The criterion above is rather restrictive, as it, for example, implies $\acl_L = \dcl_L$.
Whenever this equality holds, one can choose $L' := L \cup \set{\text{``all $L$-definable functions''}}$ to satisfy the criterion above, but this is a very artificial language.
In practice, we will only apply Theorem \ref{theorem_qe} in the case $T = \TKvs$ and in the case where $T$ is a linear o-minimal expansion of the theory of ordered divisible abelian groups.

\begin{example} \label{example_kvs_to_module}
    The theory $\TKvs\theta^C$ is complete and has quantifier elimination in the language of $R_C$-modules, i.e., $\LRC = (0, +, (r \cdot)_{r\in R_C})$.
\begin{proof}
    Completeness follows from Theorem \ref{theorem_ee_iff}, since $\cl_\theta(\varnothing) = \set{0}$.
    Quantifier elimination follows directly from Theorem \ref{theorem_qe}, with $L'$ being $L_K$, i.e., the language of $K$-vector spaces.
    Note that $\lambda \cdot = \lambda[\theta] \cdot$, so we can treat $L$ as a subset of $\LRC$.
\end{proof}
\end{example}

\noindent We will study the theory $\TKvs\theta^C$ in a future paper, but for now, notice that the example above becomes particularly interesting if $R_C$ is a field:

\begin{observation} \label{observation_str_min}
    If $R_C$ is a field, then $\TKvs\theta^C$ is the theory of $R_C$-vector spaces and therefore strongly minimal.
\end{observation}

\subsection{Algebraic closure}

We will now show that the closure operator $\cl_\theta$ introduced in Definition \ref{def_cl_theater} is actually the algebraic closure in models of $T\theta^C$.

\begin{lemma} \label{lemma_acl_indope}
    Let $\mm \models T$ be a sufficiently saturated model, and let $A, B \subseteq M$ be algebraically closed sets with $A \subseteq B$.
    There is a sequence of sets $(B_i : i \in \omega)$ such that $B_i \equiv_A B$ and $B_i \cap \acl_L(B_j : j \in \omega \setminus\set{i}) = A$ hold for every $i \in \omega$.
\begin{proof}
    By (1) of Proposition 1.5 in \cite{Adl09}, we can iteratively choose sets $B_i$ such that both $B_i \equiv_A B$ and
    $
    B_i \cap \acl_L(B_j : 0 \leq j < i) = A
    $
    hold.
    Write $B_i$ as a tuple $\ub_i = (b_{i, \alpha} : \alpha < \kappa)$.
    By Ramsey and compactness, we may assume that the sequence $(\ub_i : i \in \omega)$ is indiscernible over $A$.
    Suppose there is some $i_0 \in \omega$ with $B_{i_0} \cap \acl_L(B_j : j \in \omega \setminus\set{i_0}) \supsetneq A$.
    Then there are $j_1 < \cdots < j_m$, all different from $i_0$, an index $\alpha < \kappa$, and an $L(A)$-formula $\varphi(x; \uy_1, \dots, \uy_m)$ such that
    $
    \mm \models \varphi(b_{i_0, \alpha}; \ub_{j_1}, \dots, \ub_{j_m}),
    $
    and $\varphi(\mm; \ub_{j_1}, \dots, \ub_{j_m})$ is a finite set disjoint from $A$.
    Let $q$ be the size of this finite set.
    By the construction of the sequence, the witness cannot have $j_m < i_0$.
    Thus, after setting $j_0 := -1$, there is some $k < m$ such that $j_k < i_0 < j_{k+1}$.
    By indiscernibility, we may assume that $j_{k+1} - j_k > q+1$.
    Again by indiscernibility, we have
    $
    \mm \models \varphi(b_{i, \alpha}; \ub_{j_1}, \dots, \ub_{j_m})
    $
    for every $i$ with $j_k < i < j_{k+1}$.
    By the pigeonhole principle, there are two indices $i_1 < i_2$ in $\set{j_k +1, \dots, j_{k+1}-1}$ such that $b_{i_1, \alpha} = b_{i_2, \alpha} \not\in A$.
    However, this contradicts the equality $B_{i_2} \cap \acl_L(B_j : 0 \leq j < i_2) = A$.
\end{proof}
\end{lemma}

\begin{fact}[Standard Construction, Lemma 2.18 in \cite{Chi25}]
    \label{lemma_standart_construction}
    Given a $C$-endomorphism $\theta \colon \VV \to \VV$ and a vector space $\VV' \supset \VV$ with $\dim(\VV'/\VV) \geq \aleph_0$, there exists a $C$-endomorphism $\theta' \colon \VV' \to \VV'$ extending $\theta$.
\end{fact}

\begin{fact}[Lemma 3.19 in \cite{Chi25}] \label{lemma_substructure_R_C}
    Let $(\VV, \theta)$ be $C$-image-complete, and let $U \subseteq \VV$ be non-empty and closed under addition and $R_C$, i.e., for any $r \in R_C$ and $u \in U$, we have $r^{(\VV, \theta)}(u) \in U$.
    Then $(U, \theta_{\restriction U})$ is $C$-image-complete, and for any $r \in R_C$ and $u \in U$, we obtain
    $$
    r^{(\VV, \theta)}(u) = r^{(U, \theta_{\restriction U})}(u).
    $$
\end{fact}

\begin{theorem} \label{theorem_acl}
    Assume that $T$ satisfies \Hfour{}. For any $(\mm, \theta) \models T\theta^C$ and any set $A \subseteq M$, we have $\acl_{L_\theta}(A) = \cl_\theta(A)$.
\begin{proof}
    The inclusion ``$\supseteq$'' is clear.
    Now assume $A = \cl_\theta(A)$ and $b \not\in A$.
    Using Lemma \ref{lemma_acl_indope}, we find a sufficiently saturated model $\mm' \models T(A)$ together with an $L(A)$-elementary map $\iota_i \colon \cl_\theta(Ab) \to B_i$ for each $i \in \omega$ such that
    $
    B_i \cap \acl_L(B_j : j \in \omega \setminus \set{i}) = A
    $
    holds for all $i \in \omega$.
    We define a $C$-endomorphism $\theta'$ on $\spanA{B_i \cap \VV' : i \in \omega}{K}$ by setting $\theta'(v) := \iota_i \circ \theta \circ \iota_i^{-1}(v)$ for any $v \in B_i \cap \VV'$.
    \begin{subclaim}
        The map $\theta'$ is well-defined and indeed a $C$-endomorphism.
    \begin{innerproof}
        Well-definedness follows from $B_i \cap \acl_L(B_j : j \in \omega \setminus \set{i}) = A$ and the fact that the $\iota_i \circ \theta \circ \iota_i^{-1}$'s coincide on $A$. It remains to show that $\theta'$ is a $C$-endomorphism. First suppose that $C$ is algebraic.
    Given $b = \sum_{i \in \ii}b_{i}$ with $\ii \subset \omega$ finite and $b_i \in B_i \cap \VV'$ for each $i \in \ii$, we obtain 
    $$
    \mipo(C)[\theta'](b) = \sum\nolimits_{i \in \ii} \iota_i \circ \mipo(C)[\theta] \circ \iota_i^{-1}(b_i) = 0.
    $$
    As any $b \in \spanA{ B_i \cap \VV' : i \in \omega}{K}$ can be written in this manner, $\theta'$ is a $C$-endomorphism.

    Now assume that $C$ is transcendental.
    Given $f \in \Kp{}$ with $0 < C(f) < \infty$ and $b = \sum_{i \in \ii}b_{i}$ as above with $f^{C+1}[\theta'](b) = 0$, we need to show $f^C[\theta'](b) = 0$.
    We obtain
    $$
    f^{C+1}[\theta'](b) = \sum\nolimits_{i \in \ii} \iota_i \circ f^{C+1}[\theta] \circ \iota_i^{-1}(b_i) = 0.
    $$
    As $B_i \cap \acl_L(B_j : j \in \omega \setminus \set{i}) = A$, we see that each $\iota_i \circ f^{C+1}[\theta] \circ \iota_i^{-1}(b_i)$ must lie in $A$.
    Notice that $(B_i \cap \VV', \iota_i \circ \theta \circ \iota_i^{-1})$ is $\LKThe(A \cap \VV)$-isomorphic to the structure $(\cl_\theta(Ab) \cap \VV', \theta)$, so it is $C$-image-complete by Fact \ref{lemma_substructure_R_C}.
    Applying $f^{C+1}[\theta]^{-1} \circ f^{C+1}[\theta] = \pi_{\Image(f^C)}$ and $\Id = \pi_{\Image(f^C)} + \pi_{\Ker(f^C)}$ (see Fact \ref{fact_endo_gen} and the paragraph below), we obtain 
    $$
    b_i \in \Ker(f^C[\iota_i \circ \theta \circ \iota_i^{-1}]) + \underbrace{f^{C+1}[\iota_i \circ \theta \circ \iota_i^{-1}]^{-1}(\iota_i \circ f^{C+1}[\theta] \circ \iota_i^{-1}(b_i))}_{:= a_i \in A}
    $$ for every $i \in \ii$.
 Hence, for $l \in \set{0, 1}$, we obtain $f^{C+l}[\theta'](b) = \sum\nolimits_{i \in \ii} f^{C+l}[\theta](a_i) = f^{C+l}[\theta]\big(\sum\nolimits_{i \in \ii} a_i\big)$.
    This implies $f^{C+1}[\theta]\big(\sum\nolimits_{i \in \ii} a_i\big) = 0$ and therefore $f^{C}[\theta'](b) = f^{C}[\theta]\big(\sum\nolimits_{i \in \ii} a_i\big) = 0$, as $\theta$ is a $C$-endomorphism.
    We conclude that $\theta'$ is also a $C$-endomorphism.
    \end{innerproof}
    \end{subclaim}

    We can extend $\theta'$ to a $C$-endomorphism on all of $\VV'$ using our \standartConstruction{} (Lemma \ref{lemma_standart_construction}).
    Hence $(\mm', \theta') \models T^C_\theta(A)$.
    Since $(\mm', \theta')$ embeds into an existentially closed model, we may assume, without loss, that $(\mm', \theta') \models T\theta^C(A)$.
    Notice that, for every $i$, the map $\iota_i \colon \cl_\theta(Ab) \to B_i$ is $L$-elementary, is an $L_\theta$-isomorphism, and maps $Ab$ to $A\iota_i(b)$ while fixing $A$ elementwise.
    Hence $\tp_{L_\theta}(\iota_i(b)/A) = \tp_{L_\theta}(b/A)$ by Corollary \ref{corollary_same_type}.
    Since $\iota_i(b) \neq \iota_j(b)$ for $i \neq j$, we conclude that $\tp_{L_\theta}(b/A)$ has infinitely many realizations, so $b \not\in \acl_{L_\theta}(A)$.
\end{proof}
\end{theorem}

\noindent Let $\cl_\theta^\dcl(A)$ be the smallest set that is both closed under $\dcl_L$ and $\spanA{\;}{C}$.
It seems likely to the author that the equation $\dcl_{L_\theta}(A) = \cl_\theta^\dcl(A)$ does not hold in general.
It does, however, hold if $\acl_L = \dcl_L$, as in the o-minimal case, or, more generally, if some formula defines a linear order on every model of $T$.

\begin{example} \label{example_acl}
    For any $(\VV, \theta) \models \TKvs\theta^C$ and any subset $A \subseteq \VV$, the following holds:
    $$
    \acl_{L_\theta}(A) = \dcl_{L_\theta}(A) = \spanA{A}{C} = \spanA{A}{L_{R_C}} = \cl_\theta(A).
    $$
\begin{proof}
    Since $\acl_{L} = \spanA{\;}{K}$, we see that $\spanA{\;}{\LRC}$ is the smallest set closed under both $\acl_{L}$ and $\spanA{\;}{C}$.
    Thus, by the definition of $\cl_\theta$, we have $\spanA{\;}{\LRC} = \cl_\theta$.
    We can now use Theorem \ref{theorem_acl}.
\end{proof}
\end{example}

\noindent By Observation \ref{observation_str_min}, we know that $\TKvs\theta^C$ is strongly minimal if $R_C$ is a field.
Recall that, in strongly minimal theories, $\acl$ has the \textbf{exchange property}, i.e., whenever we have $a \in \acl(Ab)\setminus \acl(A)$, then also $b \in \acl(Aa)$.
One can ask whether $\acl_{L_\theta}$ generally has the exchange property in $T\theta^C$ and is therefore a pregeometry.
Recall that whenever the algebraic closure has the exchange property, there is a well-defined notion of dimension (see e.g., Definition 1.3 in \cite{Pil88}):

\begin{fact} \label{fact_acl_dimension}
    Let $T$ be an arbitrary theory, which does not necessarily satisfy our general setting, and suppose that algebraic closure has the exchange property in $T$.
    Given $\mm \models T$, a parameter set $A \subseteq M$, and a tuple $\ua = (a_1, \dots, a_n) \in M^n$, define the \textbf{$\bm{\acl}$-dimension} of $\ua$ over $A$ as
    $$
    \dim(\ua/A) := \max \set{ |I| : I \subseteq \set{1, \dots, n} \text{ and } (a_i)_{i \in I} \text{ is $\acl$-independent over } A }.
    $$
    If $X \subseteq M^n$ is $L(A)$-definable, define
    $$
    \dim(X) := \max \set{ \dim(\ua/A) : \ua \in X^\MM },
    $$
    where $\MM \succ \mm$ is a sufficiently large monster model, and $X^\MM$ denotes the set defined by the same formula in that monster model.
    We set $\dim(\varnothing) := -\infty$.
    When the language needs to be specified, we write $\dim_L$ for the dimension computed using $\acl_L$.
    The following standard properties hold:
    \begin{enumerate}[(i)]
        \item The definition above is well-defined, i.e., $\dim(X)$ neither depends on the choice of $A$ nor $\MM$.
        \item If $X \subseteq Y$ are $L(M)$-definable sets, then $\dim(X) \leq \dim(Y)$.
        \item $L(M)$-definable bijections preserve dimension.
        \item Every infinite definable set has dimension at least $1$.
        \item Suppose that $X \subseteq M^m$ is a $L(M)$-definable set and $f \colon X \to M^n$ is a $L(M)$-definable function such that $\dim(f^{-1}(\set{f(\ua)})) = q$ is the same for every $\ua \in X$. Then $\dim(X) = \dim(f(X)) + q$.
    \end{enumerate}
\end{fact}

\noindent We assume that the reader is generally familiar with this notion of dimension. Note that it is often considered in the literature with the additional assumption that $T$ eliminates $\exists^\infty$.
This assumption is mostly used to ensure that this dimension is definable in some way. It is, however, not necessary to show the points in Fact \ref{fact_acl_dimension} above; in fact, it is straightforward to prove these points with the definition of $\dim$ and standard facts about pregeometries (see, e.g., Section C1 in \cite{TZ12}).

In Theorem \ref{theorem_acl}, we have seen that $\acl_{L_\theta} = \cl_\theta$, where $\cl_\theta$ is the smallest set closed under both $\acl_L$ and $\spanA{\;}{C}$.
This raises a simpler question: When does $\spanA{\;}{C}$, the closure under sums and the endomorphisms from $R_C$, have the exchange property? Another potential question one might ask is: Does $\acl_{L_\theta} = \cl_\theta$ have the exchange property if both $\acl_L$ and $\spanA{\;}{C}$ have the exchange property?

\begin{lemma} \label{lemma_both_im_ker_inf}
    Let $C$ be a kernel configuration such that $R_C$ is not a field.
    Then there is some $f \in \Kp{}$ such that both $\Image(f)$ and $\Ker(f)$ are infinite in every existentially closed model of $T_\theta^C$.
    Moreover, given any infinite $L(M)$-definable subset $X \subseteq \VV$, the sets $\Ker(f) \cap X$ and $\Image(f) \cap X$ are both infinite.
\begin{proof}
    Work in an existentially closed model $(\mm, \theta)$ of $T_\theta^C$.
    By Fact \ref{fact_when_field}, there is some $f \in \Kp{}$ with $C(f) \neq 0$.
    Using Theorem \ref{theorem_big_characterization}, it is easy to check that 
    $$
    (\mm, \theta) \models \exists x \in \VV : x \in X \setminus U \wedge f[\theta](x) = 0
    $$
    holds for any finite $U \subseteq \VV$.
    Hence $\Ker(f) \cap X$ is infinite.
    Fact \ref{fact_when_field} also tells us that either:
    \begin{enumerate}[(i)]
        \item $C$ is a transcendental kernel configurations.
        In this case, one can use Theorem \ref{theorem_big_characterization} to show that $(\mm, \theta) \models \exists x \in \VV : f[\theta](x) \in X \setminus U$ holds for any finite $U \subseteq \VV$.
        \item $C$ is algebraic and $f \mid \mipo(C) \neq f$.
        In this case, either $\Image(f) \cap X \supseteq \Ker(g) \cap X$ is infinite if there is another $g \in \Kp{} \setminus \set{f}$ with $g \mid \mipo(C)$, or $C(f) \geq 2$ otherwise.
        If $C(f) \geq 2$, then $(\mm, \theta) \models \exists x \in \VV : x \in \VV \wedge f[\theta](x) = u$ holds for any $u$ in the infinite set $\Ker(f) \cap X$, again by Theorem \ref{theorem_big_characterization}.
    \end{enumerate}
    In both cases, we conclude that $\Image(f)$ is infinite.
\end{proof}
\end{lemma}

\begin{theorem} \label{theorem_nes_cond_fixed}
    The closure operation $\spanA{\;}{C}$ from Definition \ref{def_span_c} has the exchange property in the theory $\TKvsThe \cup \set{\text{``$\theta$ is $C$-image-complete''}}$ if and only if $R_C$ is a field. Moreover, if $T$ satisfies \Hfour{}, the following holds:
    \begin{enumerate}[(i)]
        \item If $\acl_L$ does not have the exchange property in $T$, then $\acl_{L_\theta}$ does not have the exchange property in $T\theta^C$.
        \item Assume that $C$ is non-trivial and that $\acl_L$ has the exchange property in $T$. If $\VV$ is not one-dimensional (with respect to the dimension induced by $\acl_L$), then $\acl_{L_\theta}$ does not have the exchange property in $T\theta^C$.
        \item If $\spanA{\;}{C}$ does not have the exchange property in $\TKvsThe \cup \set{\text{``$\theta$ is $C$-image-complete''}}$, then $\acl_{L_\theta}$ does not have the exchange property in $T\theta^C$.
    \end{enumerate}
\begin{proof}
    We first show that $\spanA{\;}{C}$ has the exchange property if $R_C$ is a field.
    Suppose that there is an element
    $
    v \in \spanA{u_0, \dots, u_{n}}{C} \setminus \spanA{u_1, \dots, u_{n}}{C}.
    $
    By Definition \ref{def_span_c}, we obtain $r_0, \dots, r_n \in R_C$ with $r_0 \neq 0$ such that
    $$
    v = r_0(u_0) + \cdots + r_n(u_n).
    $$
    Since $R_C$ is a field, we have
    $
    u_0 = r_0^{-1}\big(v - r_1(u_1) - \cdots - r_n(u_n)\big),
    $
    and hence $u_0 \in \spanA{v, u_1, \dots, u_n}{C}$. The other direction follows with (iii) applied with $T = \TKvs$.

    We now prove (i). If $\acl_L$ does not have the exchange property in $T$, we can find a model of $T$ that witnesses this. Without loss, we can assume that this model is contained in a sufficiently large monster model $(\MM, \theta) \models T\theta^C$. Now there are a finite set $A \subsetneq M$, and $a, b \in M$ such that
    $$
    a \in \acl_L(Ab) \setminus \acl_L(A)
    \quad\text{and}\quad
    b \not\in \acl_L(Aa).
    $$
    As $a \not\in \acl_L(A)$, we can find some $a' \in M$ with $\tp_L(a'/A) = \tp_L(a/A)$ and $a' \not\in \acl_{L_\theta}(A)$. Similarly, as $b \not\in \acl_L(Aa)$, we can find some $b' \in M$ with $\tp_L(a'b'/A) = \tp_L(ab/A)$ and $b' \not\in \acl_{L_\theta}(Aa')$. This construction clearly yields $a' \in \acl_{L_\theta}(Ab') \setminus \acl_{L_\theta}(A)$ and $b' \not\in \acl_{L_\theta}(Aa')$, so $\acl_{L_\theta}$ does not have the exchange property.
    
    Next, we show (ii).
    Work in a sufficiently large monster model $(\MM, \theta) \models T\theta^C$ and assume that $\VV$ has $\acl_L$-dimension $m \geq 2$ and arity $n \geq m$. This means (after possibly permuting the coordinates of $\VV$) that there is $(a_1, \dots, a_n) \in \VV$ such that $a_1, \dots, a_m$ are $\acl_L$-independent over $\varnothing$ and $a_{m+1}, \dots, a_n \in \acl_L(a_1, \dots, a_m)$.
    Equivalently, $a_i \not\in \acl_L(a_1, \dots, a_{i-1}, a_{i+1}, \dots, a_m)$ holds for every $i \in \set{1, \dots, m}$, and there is an $L$-formula $\varphi(y_{m+1}, \dots, y_{n}; y_1, \dots, y_m)$ algebraic in $(y_{m+1}, \dots, y_n)$ such that 
    $$
    \MM \models \varphi(a_{m+1}, \dots, a_{n}; a_1, \dots, a_m).
    $$
    We can furthermore assume that there is $q \in \NN$ such that $|\varphi(\MM; d_1, \dots, d_m)| \leq q$ for any $d_1, \dots, d_m \in M$.
    Define $A := \set{a_1, \dots, a_{m-1}}$ and set 
    \begin{align*}
        \psi(x; y) := x \in \VV \wedge \exists (y_{m}, \dots, y_n) &: \big(x = (y, a_2, \dots, a_{m-1}, y_m, \dots, y_n) \\ & \hspace{70pt}\wedge \varphi(y_{m+1}, \dots, y_n; y, a_2, \dots, a_{m-1}, y_m)\big).
    \end{align*}
    Since $a_m \not\in \acl_L(a_1, \dots, a_{m-1})$, and $\MM \models \varphi(a_{m+1}, \dots, a_{n}; a_1, \dots, a_m)$, we see that the $L(A)$-definable set
    $
    B_0 := \psi(\MM; a_1)
    $
    is infinite and that the preimage of any element under the $L(A)$-definable map $\pi_{m \restriction B_0}$ has cardinality at most $q$ (as always $\pi_m$ is the projection to the $m$-th coordinate). This means that, given any $b \in M$, we have $(\pi_{m \restriction B_0})^{-1}(\set{b}) \subseteq \acl_L(Ab)$. This also means that, given any finite set $U \subseteq M$, the formula $x \in B_0 \wedge \pi_m(x) \not\in U$ also defines an infinite subset of $\VV$, since at most $q \cdot |U|$ elements are excluded from the infinite set $B_0$.
    
    Since we have $a_1 \not\in \acl_L(a_2, \dots, a_m)$, we can find some element $a \in M \setminus \acl_{L_\theta}(A)$ that satisfies $\tp_L(a/a_2, \dots, a_m) = \tp_L(a_1/a_2, \dots, a_m)$. With this, one can easily verify that, for $B_1 := \psi(\MM; a)$, the $L(M)$-formula $x \in B_1$ defines an infinite subset of $\VV$. 
    \begin{subclaim} \label{claim_claim_claim_claim_and_claim}
    There is some $v \in \VV$ such that
    $
    v \in B_0 \wedge \pi_m(v) \not\in \acl_{L_\theta}(Aa) \wedge \theta(v) \in B_1
    $.
    \begin{innerproof}
        By compactness, we only need to show that the formula $x \in B_0 \wedge \pi_m(x) \not\in U \wedge \theta(x) \in B_1$ has a realization for every finite $U \subseteq \acl_{L_\theta}(Aa)$. Fix such a $U$ throughout the proof. Note that any $L(M)$-formula in a single variable that defines an infinite subset of $\VV$ implies no finite disjunction of non-trivial linear dependencies over $\VV$. Because $x^0$ and $x^1$ are different variables, it is easy to see that
        $$
        x^0 \in B_0 \wedge \pi_m(x^0) \not\in U \wedge x^1 \in B_1 
        $$
        implies no finite disjunction of non-trivial linear dependencies in $\xvec$ over $\VV$.

        If $C$ is transcendental, we can simply apply Theorem \ref{theorem_big_characterization}, since $x^0 \in B_0 \wedge \pi_m(x^0) \not\in U \wedge x^1 \in B_1$ implies no finite disjunction of non-trivial linear dependencies in $\xvec$ over $\VV$ and $S(x) := \top$ is a $C$-sequence-system over $(\VV, \theta)$ (that bounds every formula).

        If $C$ is algebraic, set $F := \Kp{0<C<\infty}$, set $\ux := (x_f : f \in F)$, and define $\rho_f \in K[X]$ as the remainder of $X$ divided by $f^C$ for all $f \in F$. Since $C$ is non-trivial, $\mipo(C) = \prod_{f \in F} f^C$ has degree greater than $1$. This implies that either $|F| \geq 2$ or that there is $f \in F$ with $\deg(f^C) \geq 2$. Since $x^0 \in B_0 \wedge \pi_m(x^0) \not\in U \wedge x^1 \in B_1$ implies no finite disjunction of non-trivial linear dependencies in $\xvec$ over $\VV$, we see that in any case
        $$
        \psi(\uxvec) := \sum\nolimits_{f \in F} x^0_f \in B_0 \wedge \pi_m\Big(\sum\nolimits_{f \in F} x^0_f\Big) \not\in U \wedge \Big(\sum\nolimits_{f \in F} \sum\nolimits_{i=0}^{\deg(\rho_f)} (\rho_f)_i \cdot x^i_f \Big) \in B_1
        $$
        implies no finite disjunction of non-trivial linear dependencies in $\uxvec$ over $\VV$. Because the conjunction $S(\ux) := \bigwedge_{f\in F} f^C[\theta](x_f) = 0$ is clearly a $C$-sequence-system over $(\VV, \theta)$ that bounds $\psi(\uxvec)$, we can use Theorem \ref{theorem_big_characterization} to show that $\psi_\theta(\ux) \wedge S(\ux)$ has a realization $(v_f : f \in F) \in \VV$. One can now check that $\sum_{f \in F} v_f$ realizes $x \in B_0 \wedge \pi_m(x) \not\in U \wedge \theta(x) \in B_1$.
    \end{innerproof}
    \end{subclaim}
    \noindent By the definitions of $B_0$, $B_1$, and $v$, we have $b := \pi_m(v) \in M \setminus \acl_{L_\theta}(Aa)$ and $\pi_1(\theta(v)) = a$. As established above we have $(\pi_{m \restriction B_0})^{-1}(\set{b}) \subseteq \acl_L(Ab) \subseteq \acl_{L_\theta}(Ab)$. With this, we obtain 
    $$
    a \in \pi_1(\theta((\pi_{m\restriction B_0})^{-1}(\set{b}))) \subseteq \acl_{L_\theta}(Ab).
    $$
    As we also have $a \not\in \acl_{L_\theta}(A)$ and $b \not\in \acl_{L_\theta}(Aa)$ by construction, we see that $\acl_{L_\theta}$ does not have the exchange property in $T\theta^C$.

    Finally, we show (iii). By (i) and (ii), we can assume that $\acl_L$ has the exchange property in $T$ and that $\VV$ is a $1$-dimensional set. Assume further that $R_C$ is not a field and, toward a contradiction, that $\acl_{L_\theta}$ has the exchange property in $T\theta^C$. By Lemma \ref{lemma_both_im_ker_inf}, there is some $f \in \Kp{0<C}$ for which both $\Ker(f)$ and $\Image(f)$ are infinite. 
    Notice that for any $u \in \Image(f)$, the preimage under $f[\theta]$ is in $L_\theta(M)$-definable bijection with $\Ker(f)$. Using (iii), (iv) and (v) of Fact \ref{fact_acl_dimension}, we obtain
    $$
    1 = \dim(\VV) \geq \dim_{L_\theta}(\VV) = \dim_{L_\theta}(\Image(f)) + \dim_{L_\theta}(\Ker(f)) \geq 1 + 1,
    $$
    which yields a contradiction.
    Note that the first inequality follows from the definition of the $\acl$-dimension, as being $\acl_{L_\theta}$-independent over a set also implies being $\acl_{L}$-independent over the same set.
\end{proof}
\end{theorem}

\noindent The next example, with $C$ chosen such that $R_C$ is a field, shows that there are cases where both $\acl_L$ and $\spanA{\;}{C}$ have the exchange property and $\VV$ is a $1$-dimensional set, but $\acl_{L_\theta}$ does not have it in $T\theta^C$.

\begin{example} \label{ex_no_ex}
    Consider the theory $\RCF$ of real closed fields with $K = \QQ$ and
    $$
    (\VV, 0, +, (q\cdot)_{q \in \QQ}) = (R_{>0}, 1, \cdot, (x \mapsto x^q)_{q \in \QQ}).
    $$
    In the theory $\RCF\theta^C$, the algebraic closure $\acl_{L_\theta} = \cl_\theta$ does not have the exchange property (unless $C$ is trivial).
\begin{proof}
    For any $\rr \models \RCF$ and $u \in R_{>0}$, one easily verifies that the formula $x^0 + x^1 = u$, where $+$ is the actual addition on $\rr$, does not imply any finite disjunction of non-trivial linear dependencies in $\xvec$ over $\VV$.
    In this case, these are multiplicative dependencies.
    With arguments similar to the proof of Claim \ref{claim_claim_claim_claim_and_claim}, we see, for a fixed model $(\rr, \theta) \models \RCF\theta^C$ and $u \in R_{>0} \setminus \acl_{L_\theta}(\varnothing)$, that the set $\set{v \in R_{>0} : v + \theta(v) = u}$ is infinite.
    Hence there are $(\rr', \theta') \succ (\rr, \theta)$ and $v' \in R'_{>0}$ with $v' \not\in R$ and $v' + \theta(v') = u$.
    We conclude that $u \in \acl_{L_\theta}(v') \setminus \acl_{L_\theta}(\varnothing)$ and $v' \not\in \acl_{L_\theta}(u)$.
\end{proof}
\end{example}

\section{o-minimal Open Core} \label{sec_o_min_open_core}

As noted in Example \ref{examples_hfour}, the results in \cite{Blo23} yield many o-minimal theories that satisfy \Hfour{}.
In particular, any o-minimal expansion of $\RCF{}$ with $(\VV, +, 0, (q\cdot)_{q\in \QQ})$ being the $\QQ$-vector space induced by the multiplication on the positive elements satisfies \Hfour{} if and only if no partial exponential function is definable.
Suppose that $T'$ is an expansion of an o-minimal theory. If $T'$ is not o-minimal, then one often asks the question whether $T'$ has an o-minimal open core in the following sense:

\begin{definition} \label{def_o_min_open_cire}
    We define the following:
    \begin{enumerate}[(i)]
        \item Given an expansion $\mm = (M, <, \dots )$ of a dense linear order without endpoints, we define the \textbf{open core} of $\mm$, denoted $\mm^\circ$, as its domain $M$ together with a predicate for the closure (with respect to the topology induced by the order) of each $M$-definable set (of arbitrary arity).
        \item Given two theories $T_0, T_1$ extending the theory of dense linear orders without endpoints and with $T_0$ being o-minimal, we say that $T_0$ is an \textbf{o-minimal open core} of $T_1$, if for every $\mm \models T_1$ there is $\mm' \models T_0$ with the same domain and order $(M, <)$ such that every set definable in $\mm^\circ$ with parameters is also definable in $\mm'$ with parameters and vice versa. 
    \end{enumerate}
\end{definition}

\noindent Since $T\theta^C$ is an expansion of $T$, a natural question is: If $T$ is o-minimal, does $T\theta^C$ have an o-minimal open core?
In this section, we answer this question.
First, note that $T\theta^C$ itself is not o-minimal:

\begin{observation} \label{rem_ofc_not_o_min}
    The theory $T\theta^C$ is not o-minimal, unless $C$ is trivial.
\begin{proof}
    Work in a model $(\mm, \theta) \models T\theta^C$.
    Let $U_0$ and $U_1$ be arbitrary non-empty definable open subsets of $\VV$.
    One shows that the formula $x_0 \in U_0 \wedge x_1 \in U_1$ implies no finite disjunction of non-trivial linear dependencies in $x_0x_1$ over $\VV$ and is bounded by some $C$-sequence-system $S$.
    Using arguments as in Claim \ref{claim_claim_claim_claim_and_claim}, one can prove that the formula $x \in U_0 \wedge \theta(x) \in U_1$ is consistent.
    By varying $U_0$ and $U_1$, one shows that the graph of $\theta$ is dense and co-dense in $\VV^2$.
    This contradicts o-minimality.
\end{proof}
\end{observation}

\noindent Let $T$ be model-complete and satisfy \Hfour{} as always.
In this section, we additionally assume that $T$ is o-minimal and an expansion of the theory of dense linear orders without endpoints.
Note that the vector space $(\VV, 0, +, (\lambda \cdot)_{\lambda \in K})$ can still be defined on an $n$-ary set.
The topology on $M^m$ for any $m \geq 1$ is generated by products of intervals, i.e., sets of the form
$$
B(\uepsilon) := \Set{(a_1, \dots, a_m) \in M^m : \mm \models \bigwedge\nolimits_{k=1}^m \epsilon_{0,k} < a_k < \epsilon_{1,k}}
$$
with $\uepsilon = (\epsilon_{i, k} : i \in \set{0,1}, 1 \leq k \leq m) \in M^{2m}$.
Note that we do not assume that the vector space operations on $\VV$ are continuous.
The main theorem of this section is the following, where $\operatorname{Cl}$ denotes the topological closure:

\begin{theorem} \label{theorem_o_min_open_core}
    Assume that $T$ is o-minimal and satisfies \Hfour{}. Given $(\mm, \theta) \models T\theta^C$, $A \subseteq M$, and $X \subseteq M^n$ that is $A$-definable in $(\mm, \theta)$, the set $\operatorname{Cl}({X})$ is already $\cl_\theta(A)$-definable in $\mm$.
     In particular, $T$ is an o-minimal open core of the theory $T\theta^C$.
\end{theorem}

\noindent Recall that, for any $L$-formula $\psi(\ux; \uw)$, condition \Hfour{} gives us another formula $\sigma_\psi(\uw)$ such that $\sigma_\psi(\ud)$ holds if and only if $\psi(\ux; \ud)$ implies no finite disjunction of non-trivial linear dependencies over $\VV$.
If we apply \Hfour{} to a formula of the form $\exists \uz : \uz \in B(\ue) \wedge \psi(\ux;\uz; \uw)$, then the resulting formula is a priori a formula in both $\uw$ and $\ue$.
The following lemma shows that we can choose this formula so that it depends less on $\ue$.

\begin{lemma} \label{lemma_local_hfour}
    Let $\psi(\ux; \uz; \uw)$ be an $L$-formula, and let $\sigma(\uw\ue) := \sigma_{\phi}(\uw\ue)$ be the formula obtained from \Hfour{} (see Definition \ref{def_hfour}) for $\phi(\ux; \uw\ue) := \exists \uz : \uz \in B(\ue) \wedge \psi(\ux;\uz; \uw)$.
    Then the formula $\sigma(\uw\ue)$ is equivalent to a formula of the form
    \begin{align}
        \exists \ux \in \VV : \exists \uz : \uz \in B(\ue) \wedge \psi(\ux; \uz; \uw) \wedge \bigwedge\nolimits_{k=1}^m \neg \varphi_k\Big(\sum\nolimits_{l=1}^n \lambda_{k, l} \cdot x_l ; \uw\Big) \label{tag_stupid_fml}
    \end{align}
    where each $\varphi_k(y; \uw)$ is algebraic in $y$ and each $(\lambda_{k, 1}, \dots, \lambda_{k, n})$ lies in $K^n \setminus \set{\uzero}$.
\begin{proof}
    Given a tuple $\ulambda{} = (\lambda_1, \dots, \lambda_n) \in K^n$, we let $\ulambda{} \cdot \ux$ denote the sum $\sum_{l=1}^n \lambda_l \cdot x_l$.
    We start with a claim:
\begin{subclaim} \label{ref_claim_imply_when_continuous}
        Let $\psi'(\ux; \uz; \uw)$ be an $L$-formula, let $\ulambda{}_1, \dots, \ulambda{}_m \in K^n \setminus \set{\uzero}$ be non-zero tuples in our field, and let $\varphi_1(y; \uw\ue), \dots, \varphi_m(y; \uw\ue)$ be $L$-formulas that define finite subsets of $\VV$ in $y$.
        Assume also that, for each $k \in \set{1, \dots, m}$, the function $\ux \mapsto \ulambda{}_k \cdot \ux$ is continuous on the set defined by $\exists \uz\uw : \psi'(\ux; \uz; \uw)$.

        Then the formula $\sigma'(\uw\ue) := \exists \ux \in \VV : \exists \uz : \uz \in B(\ue) \wedge \psi'(\ux; \uz; \uw) \wedge \bigwedge\nolimits_{k=1}^m \neg\varphi_k( \ulambda{}_k \cdot \ux; \uw \ue)$ is, modulo $T$, implied by a formula of the form
        $$
        \exists \ux \in \VV : \exists \uz : \uz \in B(\ue) \wedge \psi'(\ux; \uz; \uw) \wedge \bigwedge\nolimits_{k=1}^m \neg\varphi'_k( \ulambda{}_k \cdot \ux; \uw)
        $$
        where $\varphi'_1(y; \uw), \dots, \varphi'_m(y; \uw)$ are also $L$-formulas that define finite subsets of $\VV$ in $y$.
        Note that the tuples $\ulambda{}_k$ are unchanged.
    \begin{innerproof}
        Throughout this proof, we set $\psi'(\VV; \mm; \ud) := \set{\uv\ua \in \VV^{|\ux|} \times M^{|\uz|} : \mm \models \psi'(\uv;\ua;\ud)}$ and use similar notation for other formulas in $\ux\uz\uw$.
        We first use cell decomposition to find $L$-formulas $\zeta_{1, 1}(\ux; \uz; \uw), \dots, \zeta_{m, n_m}(\ux; \uz; \uw)$ such that, for any $\mm \models T$, any $\ud \in M$, and any $k \in \set{1, \dots, m}$, the formula $\bigvee_{l=1}^{n_k} \zeta_{k, l}(\ux; \uz; \ud)$ defines the set
        $$
        \set{\uv\ua \in \psi'(\VV; \mm; \ud) : \text{$\ux \mapsto \ulambda{}_k \cdot \ux$ is locally constant around $\uv\ua$}},
        $$
        and each $\zeta_{k, l}(\ux; \uz; \ud)$ either defines a definably connected component of that set or $\varnothing$.
        Here $\ux \mapsto \ulambda{}_k \cdot \ux$ is locally constant around $\uv\ua$ means that there is an open subset $V \subseteq \psi'(\VV; \mm; \ud)$ with respect to the subset topology such that $\uv\ua \in V$ and, for any $\uv'\ua' \in V$, we have $\ulambda{}_k \cdot \uv' = \ulambda{}_k \cdot \uv$.
        Being locally constant implies being constant on a definably connected component, so the formula
        $$
        \varphi'_k(y; \uw) := \exists \ux\uz : \bigvee\nolimits_{l=1}^{n_k} \zeta_{k, l}(\ux; \uz; \uw) \wedge y = \ulambda{}_k \cdot \ux
        $$
        defines a finite subset of $\VV$ in $y$ for any $k \in \set{1, \dots, m}$.

        We now show that these formulas $\varphi'_k(y; \uw)$ are as desired.
        For this, fix $\mm \models T$ and parameter tuples $\ud$ and $\uepsilon$ from $M$ such that the formula $\ux \in \VV \wedge \uz \in B(\uepsilon) \wedge \psi'(\ux; \uz; \ud)$ is consistent and $\mm \models \neg\sigma'(\ud\uepsilon)$, where $\sigma'(\uw\ue)$ is defined as in the statement of this claim.
        Then we must have
        $$
        \underbrace{\set{ \uv\ua \in \VV^{|\ux|} \!\times\! B(\uepsilon) : \psi'(\uv;\ua; \ud) }}_{=: U} = \bigcup\nolimits_{k=1}^m\!\bigcup\nolimits_{u \in \varphi_k(\VV; \ud\uepsilon)} \underbrace{\set{ \uv\ua \in \VV^{|\ux|} \!\times\! B(\uepsilon) : \psi'(\uv;\ua; \ud) \wedge \ulambda{}_k \cdot \uv = u }}_{=: U_{k, u}}\!.
        $$
        The set $U$ is an open subset of $\psi'(\VV; \mm; \ud)$, and every $U_{k, u}$ is a closed subset of $U$, since the function $\ux \mapsto \ulambda{}_{k} \cdot \ux$ is continuous on the set defined by $\exists \uz\uw : \psi'(\ux; \uz; \uw)$.
        Since $\varphi_k(\VV; \ud\uepsilon)$ is finite for each $k$, the union on the right-hand side is finite.
        Using o-minimality, we see that for any non-empty open subset $U' \subseteq U$, there must be some $U_{k, u}$ such that $U' \cap U_{k, u}$ has interior in $U'$.
        It follows that any $\uv\ua \in U$ must already lie in $\operatorname{Cl}(\operatorname{Int}(U_{k, u}))$ for some $k \in \set{1, \dots, m}$ and $u \in \varphi_k(\VV; \ud\uepsilon)$, where $\operatorname{Cl}$ and $\operatorname{Int}$ denote the closure and interior operations with respect to the subset topology on $\psi'(\VV; \mm; \ud)$.
        Since $\ux \mapsto \ulambda{}_k \cdot \ux$ is locally constant on $\operatorname{Int}(U_{k, u})$, every $\uv\ua \in U$ lies in the closure of $\zeta_{k, l}(\VV; \mm; \ud)$ for some $k \in \set{1, \dots, m}$ and $l \in \set{1, \dots, n_k}$.
        Finally, our continuity assumption on $\ux \mapsto \ulambda{}_k \cdot \ux$, together with the fact that $\ux \mapsto \ulambda{}_k \cdot \ux$ is constant on $\zeta_{k, l}(\VV; \mm; \ud)$, ensures that there is some $\uv'\ua' \in \zeta_{k, l}(\VV; \mm; \ud)$ with $\ulambda{}_k \cdot \uv = \ulambda{}_k \cdot \uv'$, which implies $\ulambda{}_k \cdot \uv \in \varphi'_k(\VV; \ud)$.
    \end{innerproof}
    \end{subclaim}
    \noindent Let $\sigma(\uw\ue)$ be defined as in the statement of Lemma \ref{lemma_local_hfour}.
    By Lemma \ref{lemma_hfour_fml}, we can assume
    $$
    \sigma(\uw \ue) = \exists \ux \in \VV : \exists \uz : \uz \in B(\ue) \wedge \psi(\ux; \uz; \uw) \wedge \bigwedge\nolimits_{k=1}^m \neg\varphi_k(\ulambda{}_k \cdot \ux; \uw \ue)
    $$
    with each $\ulambda{}_k = (\lambda_{k, 1}, \dots, \lambda_{k, n}) \in K^n \setminus \set{\uzero}$ and with each $\varphi_k(y; \uw \ue)$ defining a finite subset of $\VV$ in $y$.
    Using cell decomposition, we obtain $L$-formulas $\psi_1(\ux; \uz; \uw), \dots, \psi_q(\ux; \uz; \uw)$ such that $\psi(\ux; \uz; \uw) \equiv \bigvee_{i=1}^q \psi_i(\ux; \uz; \uw)$ and, for all $k \in \set{1, \dots, m}$ and $i \in \set{1, \dots, q}$, the map $\ux \mapsto \ulambda{}_k \cdot \ux$ is continuous on the set defined by $\exists \uz\uw : \psi_i(\ux; \uz;\uw)$.
    Applying Claim \ref{ref_claim_imply_when_continuous}, we see that, for all $i \in \set{1, \dots, q}$, the formula
    $$
    \sigma_i(\uw\ue) := \exists \ux \in \VV : \exists \uz : \uz \in B(\ue) \wedge \psi_i(\ux; \uz; \uw) \wedge \bigwedge\nolimits_{k=1}^m \neg\varphi_k(\ulambda{}_k \cdot \ux; \uw \ue )
    $$
    is implied by a formula of the form
    $
    \exists \ux \in \VV : \exists \uz : \uz \in B(\ue) \wedge \psi_i(\ux; \uz; \uw) \wedge \bigwedge\nolimits_{k=1}^m \neg\varphi'_{i,k}(\ulambda{}_k \cdot \ux; \uw )
    $
    where each $\varphi'_{i,k}(y; \uw)$ defines a finite subset of $\VV$ in $y$.
    Now the formula
    \begin{align}
        \exists \ux \in \VV : \exists \uz : \uz \in B(\ue) \wedge \psi(\ux; \uz; \uw) \wedge \bigwedge\nolimits_{k=1}^m \neg\Big(\bigvee\nolimits_{i=1}^q \varphi'_{i,k}(\ulambda{}_k \cdot \ux ; \uw)\Big) \label{tag_new_fml_ez}
    \end{align}
    clearly implies $\bigvee_{i=1}^q \big(\exists \ux \in \VV : \exists \uz : \uz \in B(\ue) \wedge \psi_i(\ux; \uz; \uw) \wedge \bigwedge\nolimits_{k=1}^m \neg\varphi'_{i,k}(\ulambda{}_k \cdot \ux ; \uw)\big)$, and therefore also $\bigvee_{i=1}^q \sigma_i(\uw\ue)$, which is equivalent to $\sigma(\uw\ue)$.
    On the other hand, given $\mm \models T$ and parameter tuples $\ud$ and $\uepsilon$ from $M$, if $\mm \models \sigma(\ud\uepsilon)$, then, by definition, the formula $\exists \uz : \uz \in B(\uepsilon) \wedge \psi(\ux; \uz; \ud)$ has a realization in some elementary extension that is linearly independent over $\VV$.
    Therefore, $\mm \models \sigma(\ud\uepsilon)$ implies
    $$
    \mm \models \exists \ux \in \VV : \exists \uz : \uz \in B(\uepsilon) \wedge \psi(\ux; \uz; \ud) \wedge \bigwedge\nolimits_{k=1}^m \neg\Big(\bigvee\nolimits_{i=1}^q \varphi'_{i,k}( \ulambda{}_k \cdot \ux ; \ud)\Big),
    $$
    since $\bigvee\nolimits_{i=1}^q \varphi'_{i,k}( y ; \ud)$ defines a finite subset of $\VV$.
    Thus $\sigma(\uw\ue)$ is equivalent to the formula (\ref{tag_new_fml_ez}), which has the form described in the statement of Lemma \ref{lemma_local_hfour}.
\end{proof}
\end{lemma}

\begin{lemma}
    The set
    $$
    \dd_o := \Set{\psi(\uxvec; \uz; \uw) \in L : \parbox{10.1cm}{``for all models $\mm \models T$, for all $\ud \in M$, and $\uepsilon \in M$, the $L(M)$-${}$$\hspace{4pt}$formula $\uz \in B(\uepsilon) \wedge \psi(\uxvec; \uz; \ud)$ is either inconsistent or implies no\\${}$\hspace{0pt} finite disjunction of non-trivial linear dependencies in $\uxvec$ over $\VV$''}}
    $$
    is suitable and admits a splitting strategy.
\begin{proof}
    It is clear that $\dd_o$ is suitable; see Definition \ref{def_suitable_set}.
    To show that $\dd_o$ admits a splitting strategy, fix an $L$-formula $\psi(\uxvec; \uz; \uw)$ with $\ux = (x_1, \dots, x_n)$ finite.
    Given $\ud \in M$ and a tuple $\uepsilon$ from $M$, we see that $\uz \in B(\uepsilon) \wedge \psi(\uxvec;\uz;\ud)$ implies no finite disjunction of non-trivial linear dependencies in $\uxvec$ over $\VV$ if and only if $\mm \models \sigma(\ud\uepsilon)$, where $\sigma(\uw\ue)$ is the formula from Definition \ref{def_hfour} for $\exists \uz : \uz \in B(\ue) \wedge \psi(\uxvec; \uz;\uw)$.
    By Lemma \ref{lemma_local_hfour}, with $\uxvec$ instead of $\ux$, we can assume
    \begin{align}
        \sigma(\uw\ue) = \exists \uxvec \in \VV : \exists \uz : \uz \in B(\ue) \wedge \psi(\uxvec; \uz; \uw) \wedge \bigwedge\nolimits_{k=1}^m \neg\varphi_k\Big(\sum\nolimits_{l=1}^n \lambda_{k, l}[x_l]; \uw\Big) \label{tag_sigma_for_uxvec}
    \end{align}
    with $\ulambda{}_k = (\lambda_{k, 1}, \dots, \lambda_{k, n}) \in K[X]^n \setminus \set{\uzero}$ and $\varphi_k(y; \uw)$ algebraic in $y$ for each $k$.
    Note that the tuple $\ue$ does not appear in the formulas $\varphi_k(y; \uw)$.
    Here we also write $\rho[x] = \sum_{i=0}^{\deg(\rho)}(\rho)_i \cdot x^i$ for any polynomial $\rho \in K[X]$.
    %For easier notation, we write $\ulambda_k[\uxvec] := \sum\nolimits_{l=1}^n \sum\nolimits_{i=0}^{\deg(\lambda_{k, l})} (\lambda_{k, l})_i \cdot x_l^i$.

\begin{subclaim} \label{claim_formula_in_d_o}
    The formula $\psi_0(\uxvec{}; \uz; \uw) := \psi(\uxvec{}; \uz; \uw) \wedge \forall \ue : (\uz \in B(\ue) \rightarrow \sigma(\uw\ue))$ is in $\dd_o$.
\begin{innerproof}
    Suppose that the formula $\uz \in B(\uepsilon) \wedge \psi_0(\uxvec; \uz; \ud)$ is realized by $\uvvec\ua$.
        By the definition of $\psi_0(\uxvec;\uz; \uw)$, we immediately see that $\mm \models \sigma(\ud\uepsilon)$.
        Using Definition \ref{def_hfour}, there are an elementary extension $\mm' \succ \mm$, $\uvvec' \in \VV'$, and $\ua' \in M'$ such that $\uvvec{}'$ is linearly independent over $\VV$ and $\mm' \models \ua' \in B(\uepsilon) \wedge \psi(\uvvec'; \ua'; \ud)$.
        Since $\uvvec{}'$ is linearly independent over $\VV$, we obtain
        $$
        \mm' \models \bigwedge\nolimits_{k=1}^m \neg\varphi_k\Big(\sum\nolimits_{l=1}^n \lambda_{k, l}[v'_l]; \ud\Big)
        $$
        and hence $\mm' \models \sigma(\ud \uepsilon')$ for any $\uepsilon' \in M'$ with $\ua' \in B(\uepsilon')$.
        Therefore $\mm' \models \psi_0(\uvvec{}';\ua'; \ud)$.
        We conclude that the formula $\uz \in B(\uepsilon) \wedge \psi_0(\uxvec; \uz; \ud)$ is, for any $\ud \in M$, either inconsistent or implies no finite disjunction of non-trivial linear dependencies in $\uxvec$ over $\VV$.
        Hence $\psi_0(\uxvec; \uz; \uw) \in \dd_o$.
\end{innerproof}
\end{subclaim}
    \noindent Define the formula $\psi'(\uxvec; \uz; \uw) := \psi(\uxvec; \uz; \uw) \wedge \exists \ue : \big(\uz \in B(\ue) \wedge \neg\sigma(\uw\ue)\big)$.
    Clearly, we have $\psi(\uxvec; \uz; \uw) \equiv \psi_0(\uxvec; \uz; \uw) \vee \psi'(\uxvec; \uz; \uw)$.
    Looking at (\ref{tag_sigma_for_uxvec}), we see that $\neg\sigma(\uw\ue)$ is
    $$
    \forall\uxvec\in \VV : \forall \uz : \Big(\big(\uz\in B(\ue) \wedge \psi(\uxvec; \uz; \uw)\big) \rightarrow \bigvee\nolimits_{k=1}^m \varphi_k\Big(\sum\nolimits_{l=1}^n \lambda_{k, l}[x_l]; \uw\Big) \Big).
    $$
    Therefore $\uxvec{} \in \VV \wedge \psi'(\uxvec; \uz; \uw)$, where $\uxvec{} \in \VV$ means that all placeholders $x^i_k$ that appear in $\psi'(\uxvec; \uz; \uw)$ are in $\VV$, implies the disjunction $\bigvee_{k=1}^m\varphi_k\big(\sum\nolimits_{l=1}^n \lambda_{k, l}[x_l]; \uw\big)$.
    Defining the formula $\psi_k(\uxvec; \uz; \uw) := \psi'(\uxvec; \uz; \uw) \wedge \varphi_k\big(\sum\nolimits_{l=1}^n \lambda_{k, l}[x_l]; \uw\big)$ for all $k \in \set{1, \dots, m}$, we obtain
    $$
    T \models \forall\uz\uw: \forall \uxvec \in \VV :\psi(\uxvec; \uz; \uw) \leftrightarrow \bigvee\nolimits_{k=0}^m \psi_k(\uxvec; \uz; \uw).
    $$
    For $k \geq 1$, the formula $\psi_k(\uxvec; \uz; \uw)$ is as in (ii) of Definition \ref{def_suitable_set}, since it implies some linear dependency, and by Claim \ref{claim_formula_in_d_o}, the formula $\psi_0(\uxvec; \uz; \uw)$ is as in (i) there, since it is in $\dd_o$.
    Hence $\dd_o$ admits a splitting strategy by definition.
\end{proof}
\end{lemma}

\noindent So far, we have worked entirely within the theory $T$ itself.
We now use our more technical description of definable sets, Theorem \ref{theorem_big_fml_preceise}, with the set $\dd_o$ from above to prove that $T\theta^C$ has an o-minimal open core.

\begin{proof}[Proof of Theorem \ref{theorem_o_min_open_core}]
    Let the set $X$ be defined by the $L_\theta(\ud)$-formula $\phi(\uz; \ud)$ in the model $(\mm, \theta) \models T\theta^C$.
    By Theorem \ref{theorem_big_fml_preceise}, and since closure commutes with finite unions, we may assume without loss that $X$ is defined by a formula of the form
    $$
    \exists \uy \in Y_\ud : \exists \ux \in \VV : \psi_\theta(\ux; \uz; \uy\ud) \wedge S(\ux; \pi(\uy))
    $$
    where $S(\ux; \uy)$ is a parametrized $C$-sequence-system, $\psi(\uxvec; \uz; \uy\uw) \in \dd_o$ is bounded by $S$, $Y$ is an algebraic pattern in $\uw$, and $\pi$ is a projection such that $(Y, \pi)$ is compatible with $S$.
    As $(Y, \pi)$ is compatible with $S$, the tuple $\pi(\uu)$ is compatible with $S$ for any $\uu \in Y_\ud$; see Definition \ref{def_alg_pat_comp}.
    Since the set $Y_\ud \subseteq \cl_\theta(\ud)$ is finite, we may again assume without loss that $X$ is defined by a formula of the form
    $$
    \exists \ux \in \VV : \psi'_\theta(\ux;\uz; \ud') \wedge S'(\ux)
    $$
    where $\psi'(\uxvec;\uz; \uw') \in \dd_o$, $\ud' \in \cl_\theta(\ud)$, and $S'(\ux)$ is a $C$-sequence-system over $(\VV, \theta)$.
    Define the formula $\zeta(\uz; \uw') := \exists \uxvec \in \VV : \psi'(\uxvec; \uz; \uw')$.
    It is clear that $X \subseteq \zeta(\mm; \ud')$, and therefore also $\operatorname{Cl}({X}) \subseteq \operatorname{Cl}({\zeta(\mm; \ud'))}$.
    Let $\ua \in \operatorname{Cl}({\zeta(\mm; \ud'))}$ and let $\uepsilon$ be a tuple from $M$ such that $\ua \in B(\uepsilon)$.
    Since $\ua \in \operatorname{Cl}({\zeta(\mm; \ud'))}$, the formula $\uz \in B(\uepsilon) \wedge \psi'(\uxvec; \uz; \ud')$ is consistent.
    Since $\psi'(\uxvec; \uz; \uw') \in \dd_o$, this formula implies no finite disjunction of non-trivial linear dependencies in $\uxvec$ over $\VV$.
    Notice that this formula is still bounded by $S'$.
    Thus, by Theorem \ref{theorem_big_characterization}, our characterization of existentially closed models of $T^C_\theta$, the sentence
    $$
    \exists \ux \in \VV : \exists \uz : \uz \in B(\uepsilon) \wedge \psi'_\theta(\ux; \uz; \ud') \wedge S'(\ux)
    $$
    holds in $(\mm, \theta)$.
    Hence the formula $\exists \ux \in \VV : \psi'_\theta(\ux; \uz; \ud') \wedge S'(\ux)$ has a realization in $B(\uepsilon)$.
    As this formula defines $X$ and $\uepsilon$ was arbitrary with $\ua \in B(\uepsilon)$, we obtain $\ua \in \operatorname{Cl}(X)$.
    Hence $\operatorname{Cl}(X) = \operatorname{Cl}(\zeta(\mm; \ud'))$ is $L$-definable with parameters in $\cl_\theta(\ud)$.
\end{proof}

\noindent Notice that one could likely also modify the set $\dd_o$ so that every $\psi(\uxvec; \uz; \uw) \in \dd_o$ defines a cell.
One could additionally ensure that, for all parameter tuples $\ud$ and $\uepsilon$ from $M$, the partial type $\uxvec\uz \in B(\uepsilon) \wedge \psi(\uxvec; \uz; \ud)$ is either inconsistent or implies no finite disjunction of non-trivial linear dependencies in $\uxvec$ over $\VV$.
Since $\uxvec$ is infinite, $\uepsilon$ is accordingly large, and hence $\uxvec\uz \in B(\uepsilon)$ can only be expressed as a partial type.

\bibliographystyle{alpha}
\bibliography{sample}

\begin{thebibliography}{{Blo}23}

\bibitem[Adl09]{Adl09}
Hans Adler.
\newblock A geometric introduction to forking and thorn-forking.
\newblock {\em Journal of Mathematical Logic}, 09(01):1--20, 2009.

\bibitem[{Blo}23]{Blo23}
Alexi {Block Gorman}.
\newblock Companionability characterization for the expansion of an o-minimal theory by a dense subgroup.
\newblock {\em Annals of Pure and Applied Logic}, 174(10), 2023.

\bibitem[Chi25]{Chi25}
Leon Chini.
\newblock Model theory of generic vector space endomorphisms, 2025.

\bibitem[d'E21]{dEl21b}
Christian d'Elbée.
\newblock Generic expansions by a reduct.
\newblock {\em Journal of Mathematical Logic}, 21(03):2150016, 2021.

\bibitem[Pil88]{Pil88}
Anand Pillay.
\newblock On groups and fields definable in o-minimal structures.
\newblock {\em Journal of Pure and Applied Algebra}, 53(3):239--255, 1988.

\bibitem[TZ12]{TZ12}
Katrin Tent and Martin Ziegler.
\newblock {\em A Course in Model Theory}.
\newblock Lecture Notes in Logic. Cambridge University Press, 2012.

\end{thebibliography}

\Addresses

\end{document}